\documentclass[11pt]{amsart}
\usepackage[utf8]{inputenc}
\usepackage{amsfonts}
\usepackage{amssymb}
\usepackage{amsmath}
\usepackage{delarray}
\usepackage{graphicx}
\usepackage[unicode]{hyperref}
\usepackage[all]{xy}
\newtheorem{thm}{Theorem}[section]
\newtheorem{theorem}[thm]{Theorem}

\newtheorem{lemma}[thm]{Lemma}
\newtheorem{prop}[thm]{Proposition}

\newtheorem{conjecture}[thm]{Conjecture}

\newtheorem{defn}[thm]{Definition}
\theoremstyle{remark}
\newtheorem{remark}[thm]{Remark}
\numberwithin{equation}{section}

\newcommand{\cA}{\mathcal A}
\newcommand{\cK}{\mathcal K}
\newcommand{\cN}{\mathcal N}

\newcommand{\cG}{\mathcal G}
\newcommand{\cL}{\mathcal L}
\newcommand{\cH}{\mathcal H}

\newcommand{\cO}{\mathcal O}
\newcommand{\cR}{\mathcal R}
\newcommand{\cU}{\mathcal U}

\newcommand{\cS}{\mathcal S}

\newcommand{\bbP}{\mathbb P}
\newcommand{\bbR}{\mathbb R}
\newcommand{\bbT}{\mathbb T}
\newcommand{\bbC}{\mathbb C}

\newcommand{\bbK}{\mathbb K}
\newcommand{\bbS}{\mathbb S}
\newcommand{\bbZ}{\mathbb Z}

\newcommand{\rank}{{\rm rank\ }}

\begin{document}

\newpage
%opening
\title{Geometry of integrable non-Hamiltonian systems}

\author{Nguyen Tien Zung}
\address{Institut de Mathématiques de Toulouse, UMR5219, Université Toulouse 3}
\email{tienzung@math.univ-toulouse.fr}

\date{Version 1, July 2014}
%37G05 = normal forms of dynamical systems
%53C12 = foliations
%58A17 = Pfaffian systems

\begin{abstract} 
This is an expanded version of the lecture notes for a minicourse that
I gave at a summer school called ``Advanced Course on Geometry and Dynamics of
Integrable Systems'' at CRM Barcelona, 9--14/September/2013.
In this text we study the following aspects of integrable non-Hamiltonian systems:
local and semi-local normal forms and associated torus actions for integrable systems, 
and the geometry of integrable systems of type $(n,0)$. Most of the results presented in this
text are very recent, and some theorems in this text are even original in the sense that
they have not been written down explicitly elsewhere.
\end{abstract}

\maketitle 

{\small \tableofcontents }

\section{Introduction}

This text is an expanded version of the lecture notes for a minicourse that
I gave at a summer school called ``Advanced Course on Geometry and Dynamics of
Integrable Systems'' organized by Vladimir Matveev, Eva Miranda and Francisco Presas
at Centre de Recerca Matemàtica (CRM) Barcelona, 9--14/September/2013. The aim of this
minicourse is to present some geometrical aspects of integrable \emph{non-Hamiltonian}
systems. Here the adjective non-Hamiltonian does not mean that the systems in question
cannot be Hamiltonian, it simply means that we consider general dynamical systems which 
may or may not admit a Hamiltonian structure, and even when they
are Hamiltonian we can sometimes forget about their Hamiltonian nature.

The notion of integrability for Hamiltonian systems can be traced back to a paper of
Joseph Liouville in 1855 \cite{Liouville-Torus1855}, some 160 years ago. Compared to that,
the similar notion of integrability for more general dynamical (non-Hamiltonian) systems
is very young. Though integrable non-Hamiltonian systems, especially non-holonomic systems,
have been studied since more than a century ago by people like Chaplygin, Suslov, Routh, etc., 
it is only in the 1990s that people started writing explicitly about the notion of integrability 
in the non-Hamiltonian context. Some of the earlier papers on this subject are 
Fedorov and Kozlov 1995 \cite{FedorovKozlov-Suslin}, Bogoyavlenskij 1998 \cite{Bogoyavlenskij-Integrability1998}, 
Dragovic, Gajic and Jovanovic 1998 \cite{DGJ-Nonholonomic1998}
Bates and Cushman 1999 \cite{BaCu-Nonholonomic1999}, 
Stolovitch 2000 \cite{Stolovitch-Singular2000}, 
Cushman and Duistermaat 2001 \cite{CuDu-Focus2001}, etc.
Many more articles on integrable non-Hamiltonian systems have appeared recently, including a whole
issue of the journal Regular and Chaotic Dynamics \cite{Borisov-IntegrableRCD2012},
with papers by Borisov, Bolsinov, Mamaev, Kilin, Kozlov
and other people. These papers provide a lot of interesting examples of integrable non-Hamiltonian
systems, from the more simple (e.g. rolling disks, rolling balls, Chinese tops, Chaplygin's skates)
to the more complicated ones.

Though different people arrive at integrable non-Hamiltonian systems from different points of
view, they converge at the following definition, which was probably first written down explicitly by 
Bogoyavlenskij \cite{Bogoyavlenskij-Integrability1998} who called it \emph{broad integrability}: 
a dynamical system is called integrable if it admits a sufficiently large family of commuting 
vector fields and common first integrals. More precisely, we have:

\begin{defn} \label{defn:Integrable}
A vector field $X$ on a manifold $M$ is said to be \textbf{integrable of type $(p,q)$}, 
where $p \geq 1,  q \geq 0, p+q =\dim M$, if there exist $p$ vector fields $X_1 = X, X_2,\hdots,X_p$ and $q$ 
functions $F_1,\hdots,F_q$ on $M$ which satisfy the following conditions:

i) The vector fields $X_1,\hdots,X_p$ commute pairwise:
$X$: \begin{equation}[X_i, X_j] = 0 \quad \forall i, j.\end{equation}

ii) The functions $F_1,\hdots, F_q$ are common first integrals of 
$X_1,\hdots, X_p$: \begin{equation} X_i(F_j) = 0 \quad  \forall i, j.\end{equation}

iii) $X_1 \wedge X_2 \wedge \hdots \wedge X_p \ne 0$ and $dF_1 \wedge \hdots \wedge dF_q \ne 0$ almost everywhere.

Under the above conditions, we will also say that $(X_1,\hdots,X_p, F_1,\hdots,F_q)$ 
is an {\bf integrable system of type $(p,q)$}.
\end{defn}

There are many reasons why the above definition of integrability for general dynamical systems
is the \emph{right} one. Let us list some of them here:

i) Hamiltonian systems on a symplectic $2n$-dimensional manifold which are integrable in the sense of Liouville are also
integrable of type $(n,n)$ in the above sense (i.e. $n$ commuting first integrals together with their $n$ commuting
Hamiltonian vector fields). Hamiltonian systems which are integrable in various generalized senses, 
e.g. non-commutatively integrable Hamiltonian systems in the sense of Fomenko--Mischenko \cite{FoMi-Noncommutative1978},
are also integrable in the sense of the above definition. Conversely, the \emph{cotangent lifting} of an integrable
non-Hamiltonian system is an integrable Hamiltonian system in the sense of Liouville \cite{AyoulZung-Galois2010}.

ii) Liouville's theorem \cite{Liouville-Torus1855} is still valid for integrable non-Hamiltonian systems. In particular,
 under an additional regularity and compactness condition, the movement of an integrable dynamical
 system is quasi-periodic, just like in the Hamiltonian case. 

iii) The theories of normal forms for Hamiltonian and non-Hamiltonian systems are essentially the same theory. 
In particular, an analytic integrable dynamical system always admits a local analytic normalization near a singular point,
be it Hamiltonian or not \cite{Zung-Poincare2002,Zung-Birkhoff2005}. An analytic Hamiltonian vector field will admit
a local analytic Birkhoff normalization if and only if it admits a local analytic Poincaré--Dulac normalization
when forgetting about the Hamiltonian structure \cite{Zung-Poincare2002,Zung-Birkhoff2005}.
 
iv) Many other results in the theory of integrable systems do not need the Hamiltonian structure either. In particular, 
Morales--Ramis--Simo's theorem on Galoisian obstructions to meromorphic integrability \cite{MRS-Galois2007} turns out 
to be valid also for non-Hamiltonian systems \cite{AyoulZung-Galois2010}. Various topological invariants for integrable
Hamiltonian systems, e.g. the monodromy and the Chern class (see \cite{Duistermaat-globalaction-angle1980,Zung-Integrable2003})
can be naturally extended to the case of integrable non-Hamiltonian systems as well (see, e.g., \cite{CuDu-Focus2001,BKK-Monodromy2013}).

v) Integrability in both classical and quantum mechanics means some kind of commutativity, and Definition
\ref{defn:Integrable} fits well into this philosophy. In fact, the equation $X_i(F_j) = 0$ can be rewritten as
$[X_i,f_j] = 0$ using the Schouten bracket. If one considers the functions $F_i$ as zeroth-order linear differential operators,
and the vector fields $X_i$ as first-order linear differential operators on the space of functions on the manifold $M$,
then the conditions $[X_i, X_j] = 0$ and $X_i(F_j) = 0$ mean that these $p+q$ differential operators commute, like
in the case of a quantum integrable system.

vi) Sometimes it is useful to forget about the symplectic or Poisson structure and consider integrable Hamiltonian systems on the
same footing as more general integrable dynamical systems. In particular, we will show in Subsections \ref{subsub:StructurePreserving}
and \ref{subsub:AA} a simple, short, and conceptual proof of the existence of action-angle variables using this approach.

The above points show mainly the similarities between Hamiltonian and non-Hamiltonian systems. Now let us indicate some 
differences between them:

a) The class of integrable non-Hamiltonian systems is much larger than the class of integrable Hamiltonian systems. 
There are many problems coming from control theory, economics, biology, etc. which are a priori non-Hamiltonian, 
but which can still be integrable. Not every integrable system admits a Hamiltonian structure, and
the problem of Hamiltonization is a very interesting 
and non-trivial problem, even in the case of dimension 2, see, e.g., \cite{BBM-Hamiltonization2011,ZungMinh-2dim2013}.

b) The geometry and topology of integrable non-Hamiltonian systems is much richer than that of integrable
Hamiltonian systems. In particular, there are manifolds which do not admit any symplectic structure but which admit
nondegenerate integrable non-Hamiltonian systems. Integrable non-Hamiltonian systems also admit many kinds of 
interesting singularities which are not available for Hamiltonian systems.

c) Being Hamiltonian has its advantages. For example, Noether's theorem
about the relationship between symmetries and first integrals needs the symplectic form, and there is no analogue of this
theorem in the non-Hamiltonian case. There are also machanisms of integrability which are specific to Hamiltonian systems,
e.g. the fact that bi-Hamiltonian systems are automatically integrable under some mild additional assumptions. 

d) For Hamiltonian systems, reduction (with respect to symmetry groups) commutes with integrability,  i.e. a proper symmetric 
Hamiltonian system is integrable if and only if the reduced system is integrable. On the other hand, a non-Hamiltonian system
which becomes integrable after a reduction is not necessarily integrable before reduction, 
see \cite{Zung-Torus2006,Jovanovic-Symmetries2008}. 

In this text, we will concentrate on two aspects of integrable dynamical systems that I'm most familiar with. 
Namely, in Section 2, we will study local and semi-local normal forms and associated torus actions for integrable systems, 
and in Section 3 we will study the geometry of integrable systems of type $(n,0)$. This class of systems of type $(n,0)$
is a particular but very important class in the geometric study of integrable systems, because every integrable system
becomes a system of type $(n,0)$ when restricted to a common level set of the first integrals.

\section{Normal forms and associated torus actions}
\label{section:torus}

\subsection{Regular Liouville torus actions and normal forms}
\subsubsection{Liouville's theorem}

In 1855, Liouville \cite{Liouville-Torus1855} showed that if a Hamiltonian system $X_H$ on a symplectic manifold 
$(M^{2n},\omega)$ is integrable with a momentum map ${\bf F} = (F_1,\hdots, F_n): M^{2n} \to \bbR^n$, $F_1 = H$, then each 
connected component of a compact regular level set of the momentum map ${\bf F}$ is diffeomorphic to an $m$-dimensional
torus on which the vector fields $X_H, X_{F_2},\hdots, X_{F_n}$ are constant, i.e. are 
invariant with respect to the action of the torus on itself
by translations. Each such connected level set $N$ is called a {\bf Liouville torus}. The torus $\bbT^n$ acts not only 
on $N$, but also on the nearby Liouville  tori by the same arguments, and so we have a torus $\bbT^n$-action
in a tubular neigborhood of $N$, which preserves $X_H, X_{F_2},\hdots, X_{F_n}$, and 
whose orbits are regular connected compact level sets of the momentum map. This torus action is called the {\bf Liouville
torus action} near $N$. 

Liouville's theorem can be naturally extended to the case of integrable non-Hamiltonian systems:

\begin{thm}[Non-Hamiltonian version of Liouville's theorem] \label{thm:Liouville}
Assume that $(X_1,\hdots,X_p,F_1,\hdots,F_q)$ is an integrable system of type $(p,q)$ on a manifold $M$  which
is regular at a compact level set $N$.  Then in a tubular neighborhood $\cU(N)$ there is, up to automorphisms of $\bbT^p$, 
a unique free torus action
%\begin{equation}
$\rho: \bbT^p \times \cU(N) \to \cU(N)$
%\end{equation}
which preserves the system (i.e. the action preserves each $X_i$ and each $F_j$)
and whose orbits are regular level sets of the system.  In particular, $N$ is diffeomorphic to  $\bbT^p$, and
%\begin{equation}
$\cU(N) \cong \bbT^p \times B^q$
%\end{equation}
with periodic coordinates $\theta_1 (mod\ 1),\hdots,\theta_p (mod\ 1)$ on $\bbT^p$ and coordinates
$(z_1,\hdots, z_q)$ on a $q$-dimensional ball $B^q$, such that $F_1,\hdots, F_q$ depend only on the variables
$z_1,\hdots, z_q,$ and the vector fields $X_i$ are of the type
\begin{equation} \label{eqn:XinAA}
 X_i = \sum_{j=1}^p a_{ij}(z_1,\hdots,z_q) \frac{\partial}{\partial \theta_j}.
\end{equation}
\end{thm}

The proof of the above theorem is absolutely similar to the case of integrable Hamiltonian systems on symplectic manifolds,
see, e.g., \cite{Bogoyavlenskij-Integrability1998,Zung-Torus2006}. It consists of 2 main points: 1) The map
$(F_1,\hdots,F_q): \cU(N) \to \bbR^q$ from a tubular neighborhood of $N$ to $\bbR^q$ is a 
topologically trivial fibration by the level
sets, due to the compactness of $N$ and the regularity of $(F_1,\hdots,F_q)$;
%(attention: if $(F_1,\hdots,F_q)$ is not 
%regular at $N$ then this fibration may be non-trivial and may be twisted even if the level sets are smooth); 
2) The vector fields
$X_1,\hdots,X_p$ generate a transitive action of $\bbR^p$ on the level sets near $N$, and the level sets are compact
and of dimension $p$, which imply that each level set is a $p$-dimensional compact quotient of $\bbR^p$, i.e. a torus.

Similarly to the Hamiltonian case, the regular level sets in the above theorem are called 
{\bf Liouville tori}, and the torus action is also called the {\bf Liouville torus action}. 
Theorem \ref{thm:Liouville} shows that the  flow of the vector field $X=X_1$ of an integrable system
is {\bf quasi-periodic} under some natural compactness and regularity conditions. This quasi-periodicity
is the most fundamental geometrical property of proper integrable dynamical systems.

\begin{remark}
  There is also a version of Theorem \ref{thm:Liouville} for integrable \emph{stochastic} dynamical systems, 
see \cite{ZungThien-SDS2014}. One may suspect that a version of this Liouville torus action
exists for quantum systems as well.
\end{remark}

\subsubsection{Structure-preserving property of Liouville torus actions}
\label{subsub:StructurePreserving}

The Liouville torus actions, and other associated torus actions that we will discuss later in this section,
preserve not only the system, but has the following much stronger 
{\bf structure preserving property}: roughly speaking, anything which is preserved by the system 
is also be preserved by these torus actions. In other words, one may view these torus actions as a kind of
{\bf double commutant}. More precisely, we have:

 \begin{thm}[\cite{Zung-AA2014}]  
 \label{thm:TorusPreservesStructureII}
Let $(X_1\hdots,X_p, F_1,\hdots, F_q)$ be a smooth integrable system of type $(p,q)$ on a manifold $M$ with a Liouville torus
 $N = \{F_1 = c_1, \hdots F_q = c_q\}$.
Suppose that $\cG \in \Gamma (\otimes^kTM \otimes^h T^*M)$ is a smooth tensor field which satisfies at least
one of the following two additional conditions:

i) $\cG$ is invariant with respect to the vector fields $X_1\hdots, X_p$.

ii) $\cG$ is invariant with respect to the viector field $X_1$, and moreover the orbits of $X_1$ are dense in a dense family
of Liouville tori in a tubular neighborhood of $N$. In other words, if we write $X_1 = \sum_{i=1}^p a_i(z_1,\hdots,z_q)
\partial/\partial_{\theta_i}$ in a canonical coordinate system as in Formula \ref{eqn:XinAA}, 
then for a dense family of the values of $(z_1,\hdots, z_q)$,
the numbers $a_1(z_1,\hdots,z_q), \hdots a_p(z_1,\hdots,z_q)$ are incommensurable.

Then the tensor field $\cG$ is also invariant with respect to the Liouville torus $\bbT^p$-action in a tubular neighborhood of $N$. 
\end{thm}

The tensor field $\cG$ in the above theorem can be of any nature. For example, when $\cG$ is an infinitesimal generator of a Lie group
action, we obtain that if a connected Lie group action preserves the system then it commutes with the Liouville torus action. 
When the system is osochore, i.e. preserves a volume form, then that volume form is also preserved by the Liouville torus action, etc.

\begin{proof} 
%Since the above theorem is of fundamental importance in the theory of integrable systems, and at the same time its proof
%is relatively simple, let us reproduce its proof here, following \cite{Zung-AA2014}. 
We will assume that Condition ii) is satisfied.
(The case when Condition i) is satisfied is absolutely similar). Fix a canonical coordinate
system  $(\theta_1 (mod\ 1), \hdots, \theta_p (mod\ 1), z_1,\hdots, z_q)$ in a tubular neighborhood
$\cU(N)$ of $N$ as given by Theorem \ref{thm:Liouville}. We will make a filtration of the space 
$\Gamma (\otimes^kTM \otimes^h T^*M)$ of  tensor fields of contravariant order $k$  and contravariant order $h$ as follows:

The subspace $T^{h,k}_s$ consists of sections of  $\otimes^kTM \otimes^h T^*M$ whose expression in the coordinates
$(\theta_1 (mod\ 1), \hdots, \theta_p (mod\ 1), z_1,\hdots, z_q,w_1,\hdots,w_r)$
contains only  terms of the type
\begin{equation}
\frac{\partial}{\partial \theta_{i_1}} \otimes \hdots \otimes \frac{\partial}{\partial \theta_{i_a}} \otimes
\frac{\partial}{\partial z_{j_1}} \otimes \hdots \otimes \frac{\partial}{\partial z_{i_b}} \otimes
d\theta_{i'_1} \otimes \hdots \otimes d\theta_{i'_c} \otimes
dz_{j'_1} \otimes \hdots \otimes dz_{j'_d}
 \end{equation}
with $b+c \leq s$. For example,
\begin{equation} 
 T^{h,k}_0 =  \left\{ \sum_{i,j'} f_{i,j'}\frac{\partial}{\partial \theta_{i_1}} \otimes \hdots \otimes \frac{\partial}{\partial \theta_{i_h}} \otimes
dz_{j'_1} \otimes \hdots \otimes dz_{j'_k} \right\} .
\end{equation}
Put $T^{h,k}_{-1}= \{0\}$.  It is clear that
\begin{equation}
 \{0\} = T^{h,k}_{-1} \subset T^{h,k}_ 0 \subset T^{h,k}_ 1 
\subset \hdots \subset T^{h,k}_{h+k} = \Gamma (\otimes^kTM \otimes^h T^*M).
\end{equation}
It is also clear that the above filtration is stable under the Lie derivative of the vector field $X_1$, i.e.
we have
\begin{equation}
\cL_{X_{1}}\Lambda \in T^{h,k}_s \ \ \forall  s=0,\hdots,k+h, \ \forall \Lambda \in T^{h,k}_s.
\end{equation}

Since $\cL_{X_1} \cG = 0$  by our hypothesis, and the Liouville torus action
commutes with the vector field $X_1$,  we also have that $\cL_{X_1} \overline \cG = 0,$ 
where the overline means the average of a tensor with respect to the Liouville torus action.
Thus we also have
\begin{equation}
\cL_{X_1}\hat {\cG} = 0,
\end{equation}
where 
\begin{equation}
 \hat \cG = \cG - \overline \cG
\end{equation}
has average equal to 0.

The equality
%\begin{equation}
$\cL_{X_1} \hat \cG = 0$ 
%\end{equation}
implies that the coefficients of $\hat \cG$ of the terms which are not in $T^{h,k}_{h+k-1}$, i.e. the terms of the type
\begin{equation}
\frac{\partial}{\partial z_{j_1}} \otimes \hdots \otimes \frac{\partial}{\partial z_{i_h}} \otimes
d\theta_{i'_1} \otimes \hdots \otimes d\theta_{i'_k}, 
 \end{equation}
are  invariant with respect to $X_1$. It means that these coefficient functions are constant on the orbits of $X_1$. By continuity,
it means that they are constant on Liouville tori for which the orbits of $X_1$ are dense. But since the family of such Liouville tori
is dense in the space of all Liouville tori near $N$, it implies that these functions are constant on every Liouville torus near $N$,
i.e. they are constant with respect to the Liouville $\bbT^p$-action in a neighborhood of $N$.
But any $\bbT^p$-invariant function with average 0 is a trivial function, so in fact $\hat \cG$ does not contain any term
outside of $T^{h,k}_{h+k-1},$ i.e. we have:
\begin{equation}
\hat \cG \in  T^{h,k}_{h+k-1}.
\end{equation}
By the same arguments, one can verify that if $\hat \cG \in  T^{h,k}_s$ with $s \geq 0$ 
then in fact $\hat \cG \in  T^{h,k}_{s-1}.$ So by induction we have $\hat \cG = 0,$ i.e. $\cG = \overline \cG$ is invariant
with respect to the Liouville torus action.
\end{proof}

Theorem \ref{thm:TorusPreservesStructureII} 
and its method of proof can be extended 
to the case of other invariant structures, which are similar to tensor fields, 
but which are not tensor fields strictly speaking. In particular, 
an analogue of theorem \ref{thm:TorusPreservesStructureII} for invariant
Dirac structures was in obtained \cite{Zung-AA2014}, and probably an analogue
of Theorem \ref{thm:TorusPreservesStructureII} for invariant contact structures 
(without a contact 1-form) is probably also true, with a similar proof.

In the case of linear differential operators, we also have the following result, 
whose proof is absolutely similar to the proof of Theorem \ref{thm:TorusPreservesStructureII}:

\begin{thm}[\cite{ZungThien-SDS2014}]
\label{thm:actiondiff}
Under the assumptions of Theorem \ref{thm:Liouville}, let $\Lambda$ be a linear differential operator on $M$ which satisfies at 
least one of the following two conditions :

i) $\Lambda$ is invariant with respect to $X_1,\hdots, X_p.$

ii) $\Lambda$ in invariant with respect to $X_1$, and moreover, the orbit of $X_1$ is dense in a dense family of orbits of the Liouville 
$\mathbb{T}^p$-action near $N.$

Then $\Lambda$ is invariant with respect to the Liouville $\mathbb{T}^p$-action in a neighborhood of $N.$
\end{thm}

\subsubsection{Action-angle variables}
\label{subsub:AA}

In the case of integrable Hamiltonian systems on symplectic manifolds, one can deduce easily from Theorem \ref{thm:TorusPreservesStructureII} 
the following famous theorem about the existence of action-angle variables: 

\begin{thm}
\label{thm:AA-Classical}
If $N$ is a Liouville torus of a integrable Hamiltonian system
on a symplectic manifold $(M^{2n},\omega)$ given by a momentum map ${\bf F}= (F_1,\hdots, F_n): M^{2n} \to \bbR^n$ then in a neigborhood $\cU(N)$
of $N$ there is a canonical system of coordinates $(\theta_1 (mod 1), \hdots, \theta_n (mod 1), z_1,\hdots, z_n)$, called {\bf action-angle variables}
in which the functions $F_1, \hdots, F_n$ are functions of the action variables $z_1,\hdots, z_n$ only and the symplectic structure $\omega$
has the canonical form $\omega = \sum_{i=1}^n dz_i \wedge d \theta_i$. 
\end{thm}

Remark that the above action-angle variables theorem was first proved by Henri Mineur in 1935 \cite{Mineur-AA1935,Mineur-AA1937}, though it is
oftren called Arnold--Liouville theorem. It would be more appropriate to call it {\bf Liouville--Mineur theorem}.

\begin{proof}
According to Theorem \ref{thm:TorusPreservesStructureII}, the symplectic structure is preserved by the Liouville torus $\bbT^n$-action, because it is
preserved by the Hamiltonian vector fields $X_{F_1},\hdots, X_{F_n}$ of the system. But since $N$ is a Lagrangian torus, i.e. the pull-back of $\omega$
to $N$ is trivial, the cohomology class of $\omega$ in a tubular neighborhood $\cU(N)$ of $N$ is trivial, and so this action is a Hamiltonian action 
in $\cU(N)$, i.e. it is given by a momentum map 
$(z_1,\hdots z_n): \cU(N) \to \bbR^n$. Define periodic coordinates $(\theta_1,\hdots, \theta_n)$ on $\cU(N)$ in such a way that the zero section
$S = \{\theta_1 = 0, \hdots, \theta_n = 0\}$ is a Lagrangian submanifold and $\partial / \partial \theta_i = X_{z_i}$ for all $i = 1,\hdots, n$. Then
one verifies easily that $\omega = \sum_{i=1}^n dz_i \wedge d \theta_i$ on $S$. But since both forms are $\bbT^n$-invariant, it implies that they are
equal everywhere in a neighborhood of $N$.
\end{proof}

Many generalizations of Liouville-Mineur theorem, including Nekhoroshev's theorem  
about partial action-angle variables for noncommutatively
integrable Hamiltonian systems (when the Liouville tori are isotropic instead of Lagrangian) \cite{Nekhoroshev-Integrable1972}, 
Fass\`o--Sansonetto's generalization of action-angle variables for integrable Hamiltonian systems on  
almost-symplectic manifolds (with a nondegenerated bu non-closed 2-form) \cite{FaSa-AlmostSymplectic2007},
Laurent--Miranda--Vanhaecke's generalization of action-angle variables theorem to the case of Poisson manifolds
\cite{LMV-AAPoisson2011}, and more recently, our results about action-angle variables on Dirac manifolds \cite{Zung-AA2014},
can also be deduced from Theorem \ref{thm:TorusPreservesStructureII} and its variations. 

\subsection{Local normal forms of singular points} 

\subsubsection{Poincaré--Birkhoff normal forms}

It is well-known that every smooth or analytic
vector field near an equilibrium point admits a formal Poincaré-Birkhoff normal form (Birkhoff in the
Hamiltonian case, and Poincaré-Dulac in the non-Hamiltonian case), see, e.g.,
\cite{Bruno-Local1989,Roussarie-Asterisque1975}. Let us briefly recall this
Poincaré--Birkhoff normalization theory here.

Let $X$ be a given formal or analytic vector field in a neighborhood of $0$ in
$\bbK^m$, where $\bbK = \bbR$ or $\bbC$, with $X(0) = 0$. When $\bbK = \bbR$,
we may also view $X$ as a complex vector field by complexifying it. Denote by
\begin{equation}
X = X^{(1)} + X^{(2)} + X^{(3)} + ...
\end{equation}
the Taylor expansion of $X$ in some local system of coordinates, where $X^{(k)}$ is
a homogeneous vector field of degree $k$ for each $k \geq 1$.

In the Hamiltonian case on a symplectic manifold, $X = X_H$, $m= 2n$,
$\bbK^{2n}$ has a standard symplectic structure, and $X^{(j)} =
X_{H^{(j+1)}}$, where $H^{(j+1)}$ is the term of degree $j+1$ in the Taylor expansion of $H$
in a local canonical system of coordinates.

The algebra of linear
vector fields on ${\mathbb K}^{m}$, under the standard Lie bracket, is nothing but
the reductive algebra $gl(m, \bbK) = sl(m,\bbK) \oplus \bbK$. In particular, we have
\begin{equation}
X^{(1)} = X^s + X^{nil},
\end{equation}
where $X^s$ (resp., $X^{nil}$) denotes the semi-simple (resp., nilpotent) part of
$X^{(1)}$. There is a complex linear system of coordinates $(x_j)$ in ${\mathbb
C}^{m}$ which puts $X^s$ into diagonal form:
\begin{equation}
X^s = \sum_{j=1}^m \gamma_j x_j \partial / \partial x_j ,
\end{equation}
where $\gamma_j$ are complex coefficients, called {\bf eigenvalues} of $X$
(or $X^{(1)}$) at $0$.

In the Hamiltonian case, $X^{(1)} \in sp(2n,\bbK)$ which is a simple Lie
algebra, and we also have the decomposition $X^{(1)} = X^s + X^{nil}$,
which corresponds to the decomposition
\begin{equation}
H^{(2)} = H^s + H^{nil}
\end{equation}
There is a complex canonical linear system of coordinates $(x_j,y_j)$ in ${\mathbb
C}^{2n}$ in which $H^s$ has diagonal form:
\begin{equation}
H^s = \sum_{j=1}^n \lambda_j x_j y_j ,
\end{equation}
where $\lambda_j$ are complex coefficients, called {\bf frequencies} of
$H$ (or $H^{(2)}$) at $0$.

For each natural number $k \geq 1$, the vector field $X^s$ acts linearly on the
space of homogeneous vector fields of degree $k$ by the Lie bracket, and the
monomial vector fields are the eigenvectors of this action:
\begin{equation}
[\sum_{j=1}^m \gamma_j x_j \partial / \partial x_j , x_1^{b_1}x_2^{b_2}...x_n^{b_n}
\partial / \partial x_l] = (\sum_{j=1}^n b_j\gamma_j - \gamma_l) x_1^{b_1}x_2^{b_2}...x_n^{b_n}
\partial / \partial x_l .
\end{equation}

When an equality of the type
\begin{equation}
\sum_{j=1}^m b_j\gamma_j - \gamma_l = 0
\end{equation}
holds for some nonnegative integer $m$-tuple $(b_j)$ with $\sum b_j \geq 2$, we will
say that the monomial vector field $x_1^{b_1}x_2^{b_2}...x_m^{b_m}
\partial / \partial x_l$ is a {\bf resonant term},
and that the $m$-tuple $(b_1,...,b_l - 1,..., b_l)$
is a resonance relation for the eigenvalues $(\gamma_i)$. More precisely,
a {\bf resonance relation} for  the $n$-tuple of eigenvalues $(\gamma_j)$
of a vector field $X$ is an $m$-tuple $(c_j)$ of integers satisfying the
relation $\sum c_j \gamma_j = 0,$ such that $c_j \geq -1, \sum c_j \geq
1,$ and at most one of the $c_j$ may be negative.

In the Hamiltonian case, $H^s$ acts linearly on the space of functions by
the Poisson bracket. Resonant terms (i.e. generators of the kernel of this
action) are monomials $\prod x_j^{a_j}y_j^{b_j}$ which satisfy the
following resonance relation, with $c_j = a_j - b_j$:
\begin{equation}
\sum_{j=1}^n c_j\lambda_j  = 0
\end{equation}

Denote by ${\mathcal R}$  the subset of ${\mathbb Z}^m$
(or sublattice of ${\mathbb Z}^n$ in the Hamiltonian case)
consisting of all resonance relations $(c_j)$ for a
given vector field $X$. The number
\begin{equation}
r = \dim_\bbZ (\cR \otimes \bbZ)
\end{equation}
is called the {\bf degree of resonance} of $X$. Of course, the degree of
resonance depends only on the eigenvalues of the linear part of $X$, and
does not depend on the choice of local coordinates. If $r=0$ then we say
that the system is {\bf nonresonant} at 0.

The vector field $X$ is said to be in {\bf Poincaré-Birkhoff normal form}
if it commutes with the semisimple part of its linear part:
\begin{equation}
[X,X^s] = 0.
\end{equation}
In the Hamiltonian case, the above equation can also be written as
\begin{equation}
\{H,H^s\} = 0.
\end{equation}

The above equations mean that if $X$ is in normal form then its nonlinear
terms are resonant.  A transformation of coordinates (which is symplectic
in the Hamiltonian case) which puts $X$ in Poincaré-Birkhoff normal form
is called a {\bf Poincaré-Birkhoff normalization}.

\begin{thm}[Poincaré--Dulac--Birkhoff]
Any analytic formal or vector field which vanishes at 0 admits a {\it formal} 
Poincaré-Birkhoff normalization. 
%(which does not converge in general).
\end{thm}

The proof of the above theorem is based on the classical method of step-by-step
normalization: at each step one eliminates a non-zero nonresonant monomial term 
$C_{b,l}x_1^{b_1}x_2^{b_2}...x_n^{b_n}
\partial / \partial x_l$ of lowest degree by a local coordinate transformation (diffeomorphism)
of the type $$\exp(C_{b,l}x_1^{b_1}x_2^{b_2}...x_n^{b_n}/ (\sum_{j=1}^m b_j\gamma_j - \gamma_l)),$$
where $\exp$ means the time-1 flow of the vector field. The total number of steps
is infinite in general, and the composition of all these consecutive normalizing maps
converges in the formal category but does not necessarily 
converge in the analytic category in gereral.

\subsubsection{Toric characterization of local normal forms}

Denote by ${\mathcal Q} \subset {\mathbb Z}^m$ the sublattice of ${\mathbb
Z}^m$ consisting of $m$-dimensional vectors $(\rho_j) \in {\mathbb Z}^m$ which
satisfy the following properties :
\begin{equation}
\label{eqn:Q} \sum_{j=1}^m \rho_j c_j = 0 \  \forall \ (c_j) \in {\mathcal R} \ , \
{\rm and} \ \ \rho_j = \rho_k \ \ {\rm if} \ \ \gamma_j = \gamma_k \
\end{equation}
(where $\mathcal R$ is the set of resonance relations as before). In the
Hamiltonian case, $\mathcal Q$ is defined by
\begin{equation}
\label{eqn:Q2} \sum_{j=1}^n \rho_j c_j = 0 \  \forall \ (c_j) \in {\mathcal R} .
\end{equation}

We will call the number
\begin{equation}
\label{eqn:t} d = \dim_\bbZ \mathcal Q
\end{equation}
the {\bf toric degree} of $X$ at $0$. Of course, this number depends only
on the eigenvalues of the linear part of $X$, and we have the following
(in)equality : $r + d = n$ in the Hamiltonian case (where $r$ is the
degree of resonance), and $r + d \leq m$ in the non-Hamiltonian case.

Let $(\rho^{1}_j),...,(\rho^d_j)$ be a basis of $\mathcal Q$. For each $k=1,...,
d$
define the following diagonal linear vector field $Z_k$ :

\begin{equation}
\label{eqn:Z} Z_k = \sum_{j=1}^m \rho^k_j x_j \partial / \partial x_j
\end{equation}
in the non-Hamiltonian case, and $Z_k = X_{F^k}$ where
\begin{equation}
\label{eqn:F} F^k = \sum_{j=1}^n \rho^k_j x_jy_j
\end{equation}
in the Hamiltonian case.

The vector fields $Z_1,...,Z_r$ have the following remarkable properties :

a) They commute pairwise and commute with $X^s$ and $X^{nil}$, and they are linearly
independent almost everywhere.

b) $iZ_j$ is a periodic vector field of period $2\pi$ for each $j \leq r$ (here $i =
\sqrt{-1}$). What does it mean is that if we write $iZ_j = \Re (iZ_j) + i \Im (iZ_j)
$, then $\Re (iZ_j)$ is a periodic real vector field in ${\mathbb C}^n = {\mathbb
R}^{2n}$ which preserves the complex structure.

c) Together, $iZ_1,..., iZ_r$ generate an effective linear ${\mathbb T}^r$-action in
${\mathbb C}^n$ (which preserves the symplectic structure in the Hamiltonian case),
which preserves $X^s$ and $X^{nil}$.

Another equivalent way to define the toric degree is as follows: it is the smallest number $d$
such that one can write $X^s = \sum_{j=k}^d \alpha_k Z_k$, where $\alpha_k$ are complex
coefficients and each $Z_k$ has the form 
$Z_k = \sum_{j=1}^m \rho^k_j x_j \partial / \partial x_j$
with  $\rho^k_j \in \bbZ$. The minimality of $d$ is equivalent to the fact that the numbers
$\alpha_1,\hdots, \alpha_d$ are incommensurable.

A simple calculation shows that $X$ is in Poincaré-Birkhoff normal form,
i.e. $[X,X^s] = 0$, if and only if we have
\begin{equation}
[X,Z_k] = 0 \ \ \ \forall \  k=1,...,r.
\end{equation}

The above commutation relations mean that if $X$ is in normal form, then
it is preserved by the effective $r$-dimensional torus action generated by
$iZ_1,...,iZ_r$. Conversely, if there is a torus action which preserves $
X$, then because the torus is a compact group we can linearize this torus
action (using Bochner's linearization theorem
\cite{Bochner-Linearization1945} in the non-Hamiltonian case, and the
equivariant Darboux theorem in the Hamiltonian case, see e.g.
\cite{CoDaMo-Moment1988,GuSt-Convexity1982}), leading to a normalization
of $X$. In other words, we have:

\begin{thm}[\cite{Zung-Poincare2002,Zung-Birkhoff2005}]
\label{thm:toricPB} A holomorphic (Hamiltonian) vector field $X$ in a neighborhood of $0$ in
${\mathbb C}^m$ (or $\bbC^{2n}$ with a standard symplectic form)
admits a locally holomorphic Poincaré-Birkhoff normalization if and only if it
is preserved by an effective holomorphic (Hamiltonian)
action of a real torus of dimension $t$,
where $t$ is the toric degree of $X^{(1)}$ as defined in (\ref{eqn:t}), in a
neighborhood of $0$ in ${\mathbb C}^m$ (or $\bbC^{2n}$),
which has $0$ as a fixed point and whose
linear part at $0$ has appropriate weights (given by the lattice $\mathcal Q$
defined in (\ref{eqn:Q},\ref{eqn:Q2}),
which depends only on the linear part $X^{(1)}$ of $X$).
\end{thm}

The above theorem is true in the formal category as well. But of
course, any vector field admits a formal Poincaré-Birkhoff normalization,
and a formal torus action. This (formal) torus action is intrinsically  
assocaited to the singular point, and  we will call it the {\bf associated torus
action} of the system (i.e. vector field) at the singular point.

\begin{remark}
The above toric point of view of normalization for vector fields was first developed 
in \cite{Zung-Poincare2002,Zung-Birkhoff2005}.
It was later extended to the case of diffeomorphisms (i.e. discrete-time dynamical systems) by Raissy
in \cite{Raissy-Torus2010}, and also used by Chiba \cite{Chiba-Renormalization2009} in singular perturbation and
renormalization methods. 
 \end{remark}

Theorem \ref{thm:toricPB} has many important implications. One of them is:

\begin{prop}
\label{prop:realcomplex} A real analytic vector field $X$ (Hamiltonian
or non-Hamiltonian) in the neighborhood of an equilibrium point
admits a local real analytic Poincaré-Birkhoff normalization if and only
if it admits a local holomorphic Poincaré-Birkhoff
normalization when considered as a holomorphic vector field.
\end{prop}

The proof of the above proposition (see \cite{Zung-Birkhoff2005})
is based on the fact that the complex conjugation induces an involution on the
torus action which governs the Poincaré-Birkhoff normalization. 
Recall that, even when the vector field is real, the torus acts 
not in the real space but in the complexified space in general. Only a
subtorus of this associated torus  acts in the real space. 
The dimension of this real subtorus action can be called
the {\bf real toric degree} of the system.

\subsubsection{Structure-preserving property of associated torus actions} 

The associated torus action at a singular point of a vector field has the same
{\bf struture-preserving property} as the Liouville torus actions discussed in the previous subsection:

\begin{thm}
If a formal or analytic tensor field $\cG$ is preserved by a formal or analytic vector field $X$ which vanishes
at a point $O$, then the associated torus action of $X$ at $O$ also preserves $\cG.$ 
\end{thm}

\begin{proof}(Sketch)
We can assume that $X$ is already in Poincaré--Birkhoff normal form, i.e. $[X, X^s] = 0$ where
$X^s = \sum_{k=1}^d \alpha_k Z_k$ is the semisimple linear part of $X$ at $O$, 
$Z_k = \sum_{j=1}^m \rho^k_j x_j \partial / \partial x_j$ with $\rho^k_j \in \bbZ$, 
$d$ is the toric degree, and 
$\sqrt{-1}Z_1,\hdots, \sqrt{-1}Z_d$ are the generators of the associated torus action. One verifies easily that
the Lie derivative $\cL_X$ is a linear oerator on the space of formal tensor fields of the type of $\cG$, whose
semisimple part (in the Jordan--Dunford decomposition) is $\cL_{X^s}$, so $\cL_X \cG = 0$ implies that
$\cL_{X^s}\cG = 0$, which implies that $\cL_{X^s}\cK = 0$ for every monomial term $\cK$ of $\cG$, because each
monomial tensor term is an eigenvector of the linear operator $\cL_{X^s}$. 
Thus we have $\sum_{k=1}^d \alpha_k \cL_{Z_k}\cK = 0$ for every monomial term $\cK$ of $\cG$. But since the numbers
$\alpha_1,\hdots,\alpha_d$ are incommensurable and each $\cL_{Z_k}\cK$ is an integer multiple of $\cK$, we must have
that $\cL_{Z_k}\cK = 0$ for every $k = 1,\hdots, d$.
\end{proof}

In particular, in the Hamiltonian case,  the symplectic structure is automatically preserved by the associated
torus action. Á a consequence, the existence of a analytic Birkhoff normalisation is equivalent to the existence of a Poincaré-Dulac
normalization (forgetting about the symplectic structure).

\subsubsection{Local normalization of analytic integrable systems}

When the vector field is analytically integrable, then 
Theorem \ref{thm:toricPB} leads to the following strong result about the existence 
of analytic Poincaré-Birkhoff normalization:

\begin{thm}[\cite{Zung-Poincare2002,Zung-Birkhoff2005}]
\label{thm:PBmain} Let $X$ be a local analytic vector field in $(\bbK^m,0)$, 
where $\bbK = \bbR$ or $\bbC$, such that $X(0) = 0$. 
Then $X$ admits a local analytic Poincaré-Birkhoff normalization in a
neighborhood of $0$ (which is compatible with the volume form or the symplectic structure
if $X$ is an isochore or a Hamiltonian vector field).
\end{thm}

Partial cases of the above theorem were obtained earlier by many
authors, including Rüssmann \cite{Russmann-NF1964} (the nondegenerate Hamiltonian
case with 2 degrees of freedom),
Vey \cite{Vey-Separable1978,Vey-Isochore1979} (the nondegenerate
Hamiltonian and isochore cases), Ito \cite{Ito-Birkhoff1989} (the
nonresonant Hamiltonian case), Ito \cite{Ito-Birkhoff1992} and
Kappeler et al. \cite{KaKoNe-Birkhoff1998} (the Hamiltonian case
with a simple resonance), Bruno and Walcher \cite{BrWa-NF1994} (the
non-Hamiltonian case with $m=2$). These
authors, except Vey who was more geometric, relied on long and heavy analytical
estimates to show the convergence of an infinite normalizing coordinate
transformation process. On the other hand, the proof of Theorem \ref{thm:PBmain}
in \cite{Zung-Poincare2002,Zung-Birkhoff2005} is a geometrical proof which uses 
resolution of singularities and Lojasiewicz inequalities in order to show the existence of 
an analytic torus action (i.e. to show that the associated torus action is not just 
formal but analytic), and is relatively short.

\begin{remark}
Also in the case of infinite-dimensional systems, one can talk about Poincaré-Birkhoff normalization,
and there should be an analytic associated torus action, see, e.g., \cite{KaToZu-NLS2009,KuPe-VeyInfinite2009}.
\end{remark}

\subsubsection{Toric action for a commuting family of vector fields}

It is well-known that if $X_1,\hdots, X_p$ is a family of formal or analytic 
pair-wise commuting vector fields which vanish at a point $O$, then they will admit
a simultaneous formal Poincaré--Birkhoff normalization at $O$, see; e.g., 
\cite{Stolovitch-Singular2000,Stolovitch-NF2009} and references therein. This fact corresponds to
the existence of an intrinsic formal associated torus action for the family 
$X_1,\hdots, X_p$ at $O$, whose dimension will be called the {\bf toric degree}
of the family at $O$: it is the smallest number $d$ such that we can write 
$X_i^s = \sum_{k=1}^d \alpha_{ik} Z_k$ for every $i=1,\hdots,p$, where $X_i^s$
is the semisimple linear part of $X_i$, and the vector fields $Z_k = \sum_{j=1}^m \rho^k_j x_j \partial / \partial x_j$
are like in the case of a single vector field. Another equivalent definition is 
that this torus degree is the toric degree
of a generic linear combination $\sum_{i=1}^p c_i X_i$ of the family $(X_1,\hdots, X_p).$
Using this associated torus action for a family of commuting vector fields, we get the
following simultaneous version of Theorem \ref{thm:PBmain}, whose proof
remains the same:

\begin{thm}
\label{thm:PBmain2} Let $(X_1,\hdots,X_p,F_1,\hdots,F_q)$ be a local analytic integrable system of
type $(p,q)$ in $(\bbK^m,0)$, 
where $\bbK = \bbR$ or $\bbC$, such that $X_1(0) = \hdots = X_p(0) = 0$. 
Then the vector fields $X_1,\hdots, X_p$ admit a simultaneous local analytic Poincaré-Birkhoff 
normalization in a neighborhood of $0$ (which is compatible with the volume form or the symplectic structure
if the system is isochore or Hamiltonian).
\end{thm}

\subsection{Geometric linearization of nondegenerate singular points}

\subsubsection{Nondegenerate singular points and linear systems} 

Consider an integrable system  $(X_1,\hdots, X_p, F_1,\hdots, F_q)$  
of type $(p,q)$ on a manifold $M$, and let $O \in M$ be a singular point
of the system. The number
\begin{equation}
k = \dim Span(X_1(O), \hdots, X_p(O))
\end{equation}
is called the {\bf rank} of $O$. If $k=0$ then we say that $O$ is a {\bf fixed point}.
If $k>0$, then we make a local reduction in order to obtain a
system of type $(p-k,q)$ with a fixed point as follows: without loss of generality, we can assume that
$X_1(O) \wedge \hdots \wedge X_k(O) \neq 0,$ i.e. $X_1, \hdots X_k$ generate an local free $\bbR^k$-action in a neighborhood
$\cU(O)$ of $O$. Since the $(p+q-k)$-tuple 
$(X_{k+1},\hdots, X_p, F_1,\hdots, F_p)$ is invariant with respect to this local $\bbR^k$-action, 
it can be projected to an integrable system of type $(p-k,q)$ with a fixed poin on the quotient of $\cU(O)$
by this local $\bbR^k$-action. The definition of nondegeneracy below will not depend on the choice of local reduction.

Assume now that $O$ is a fixed point.
Denote by $Y_i$ the linear part of $X_i$ at $O$, and by 
$G_j$ the homogeneous part  (i.e. the non-constant terms of lowest degree in the Taylor expansion)  
of $F_j$ in some coordinate system. The first terms of the Taylor expansion of the identities  $[X_i,X_k] = 0$ 
and $X_i(F_j) = 0$ show that the vector fields $Y_1,\hdots, Y_p$ commute with each other and have $G_1,\hdots, G_q$ as common first 
integrals. Hence, $(Y_1,\hdots,Y_p, G_1,\hdots, G_q)$ is again an integrable system of type $(p,q)$, 
which shall be called the {\bf linear part}
of the system $(X_1,\hdots, X_p, F_1,\hdots, F_q)$, provided that the independence conditions 
$Y_1 \wedge \hdots  \wedge Y_p \neq 0$ and $dG_1 \wedge \hdots \wedge dG_q \neq 0$ (almost everywhere) still hold.

\begin{defn} 
\label{defn:NondegenerateSingularPoint}
1) An integrable system $(Y_1,\hdots,Y_p, G_1,\hdots, G_q)$ of type $(p,q)$ is called 
{\bf linear} if the vector fields $Y_1,\hdots Y_p$ are linear and the functions $G_1,\hdots, G_q$ are homogeneous.
If, moreover, the linear vector fields $Y_1,\hdots Y_p$ are semisimple, then  
$(Y_1,\hdots,Y_p, G_1,\hdots, G_q)$ is called a {\bf nondegenerate linear integrable system}. \\
2) A singular point $O$ of rank $k$  of an integrable system of type $(p,q)$ 
is called {\bf nondegenerate singular point} if after a local reduction it becomes a 
fixed point of an integrable system of type $(p-k,q)$ 
whose linear part is a nondegenerate linear integrable system.
\end{defn}

\begin{remark}
If $z$ is an isolated singular point of $X_1$ in an integrable system $(X_1,\hdots, X_p, F_1,\hdots, F_q)$, 
then it will be automatically a fixed point of the system. Indeed, if  $X_i(z) \neq 0$ for some $i$, then due to the commutativity of $X_1$ with 
$X_i$, $X_1$ will vanish not only at $z$, but on the whole local  trajectory of $X_i$ which goes through $z$, and so $z$ 
will be a non-isolated singular point of $X_1$. In the definition of nondegeneracy of linear systems, we don't require the origin to be
an isolated singular point. For example, the system $(x_1 \frac{\partial}{ \partial x_1}, x_2)$ is a nondegenerate linear system of type $(1,1)$, for 
which the origin is a non-isolated singular point. 
\end{remark}

\begin{remark}
In the Hamiltonian case on a symplectic manifold, $p = q = n$, when $Y_1, \hdots, Y_n$ are linear Hamiltonian vector fields
in a canonical coordinate system, we can take $G_1,\hdots, G_n$ to be their respective quadratic Hamiltonian functions: $Y_i = X_{G_i}.$
The above definition generalizes in a natural way the well-known notion of nodegenerate singular points of
integrable Hamiltonian systems (see, e.g., \cite{Vey-Separable1978,Eliasson-Normal1990,Zung-Symplectic1996}): 
in the Hamiltonian case on a symplectic manifold with $p=q=n$, the fact that
$Y_1,\hdots,Y_n$ are linear semisimple means that they generate a Cartan subalgebra of the simple Lie 
algebra $sp(2n,\bbK)$ of linear symplectic vector fields.
It is well-known that, already in the Hamiltonian case, not every integrable
linear system is nondegenerate. 
For example, in $\mathbb{R}^4$, take $\displaystyle G_1 = x_1y_1 - x_2 y_2, G_2 = y_1y_2, 
Y_1 = x_1 \frac{\partial}{\partial x_1} - y_1 \frac{\partial}{\partial y_1} - x_2 \frac{\partial}{\partial x_2} + y_2 \frac{\partial}{\partial y_2},
Y_2 = y_2 \frac{\partial}{\partial x_1} + y_1 \frac{\partial}{\partial x_2}.$ 
Then this is a degenerate (non-semisimple) integrable
linear Hamiltonian system.
\end{remark}

\subsubsection{Nondegenerate linear systems as linear torus actions}

Consider a nondegenerate linear integrable system  $(Y_1,\hdots,Y_p, G_1,\hdots, G_q)$.
Recall that the Lie algebra of linear vector fields on $\bbK^m$, where $\bbK = \bbR$ or $\bbC$, is naturally isomorphic to the reductive
Lie algebra $gl(m,\bbK)$. Since $Y_1,\hdots, Y_p$ are commuting semisimple elements in $gl(m,\bbK)$, then they can be diagonalized
simultaneously over $\bbC.$ In other words, there is a complex coordinate system in which $Y_1,\hdots, Y_p$ are diagonal:
\begin{equation}
Y_i = \sum_{i=j}^m c_{ij} x_j\frac{\partial}{\partial x_j} .
\end{equation}
The linear independence of $Y_1, \hdots, Y_p$ means that the matrix $(c_{ij})^{i=1,\hdots,p}_{j=1,\hdots,m}$ is of rank
$p$. The set of polynomial common first integrals of $Y_1,\hdots, Y_p$ is the vector space spanned by the monomial functions
$\prod_{j=1}^m x_j^{\alpha_j}$ which satisfy the {\bf resonance equation}
\begin{equation} \label{eqn:resonance}
\sum_{j=1}^m \alpha_j c_{ij} = 0 \
\text{for all} \  i=1,\hdots, p.
\end{equation}

The set of nonnegative integer solutions of Equation (\ref{eqn:resonance}) is the intersection 
%\begin{equation}
$S \cap \mathbb{Z}^m_+,$ 
%\end{equation}
where 
%\begin{equation}
$S = \left\{ (\alpha_i) \in \mathbb{R}^m \ | \  \sum_{j=1}^m \alpha_j c_{ij} = 0 \ \text{for all} \  i=1,\hdots, p \right\}$ 
%\end{equation}
 is the $q$-dimensional
space of all real solutions of (\ref{eqn:resonance}), and $\mathbb{Z}^m_+$ is the set of nonnegative $m$-tuples of integers.
The functional independence of $G_1,\hdots,G_q$ implies that this set $S \cap \mathbb{Z}^m_+$ must have dimension 
$q$ over $\mathbb{Z}$. In particular, the set $S \cap \mathbb{R}^m_+$ has dimension $q$ over $\mathbb{R},$
and the resonance equation (\ref{eqn:resonance}) is equivalent to a linear system of equations with integer coefficients.
In other words, using a linear transformation to replace $Y_i$ by new vector fields 
\begin{equation}
Z_i = \sum_{j} a_{ij} Y_j 
\end{equation}
with an appropriate invertible matrix $(a_{ij})$ with constant coefficients, we may assume that 
\begin{equation}
Z_i = \sum_{i=j}^m \tilde{c}_{ij} x_j \frac{\partial}{ \partial x_j},
\end{equation}
where 
$\tilde{c}_{ij} = \sum_{k} a_{ik} c_{kj} \in \bbZ \ \forall \ i,j .$ The vector fields 
$\sqrt{-1}Z_1, \hdots, \sqrt{-1}Z_p$ are the generators of a linear effective $\bbT^p$-action, which is exactly the
associated torus action of the family $(Y_1,\hdots, Y_p)$ in the sense of local normal form theory. 
In particular, in this case (and for nondegenerate fixed points of integrable systems of type $(p,q)$), 
the toric degree is equal to $p$. Thus we have:

\emph{A nondegenerate linear integrable system of type $(p,q)$ is essentially 
the same as an effective linear torus $\bbT^p$-action in $\bbC^m$}.

Hence, the classification of nondegenerate linear integrable systems is essentially the same as the classification of linear torus
actions. Keep in mind that, when the system is real, the torus will act in the complexified space.
 
\subsubsection{Geometric equivalence}

Geometrically, an integrable system of type $(p,q)$ may be viewed as a singular fibration given by the level sets
of the map $(F_1,\hdots, F_q): M \to \bbK^q$, such that on each fiber there is an infinitesimal $\bbK^p$-action
generated by the commuting vector fields  $(X_1,\hdots,X_p).$
Denote by $\mathcal{F}$ the algebra of common first integrals of $X_1,\hdots,X_p.$ 
Instead of taking $F_1,\hdots, F_q$, we can choose from $\mathcal{F}$ any other family of $q$ functionally independent functions, and they will
still form with $X_1,\hdots, X_p$ an integrable system. Moreover, in general, there is no natural preferred choice of $q$ functions in
$\mathcal{F}$.  For example, consider a linear integrable 4-dimensional system of type $(1,3)$, i.e. with 1 vector field and 3 functions. The vector field is
$Y = x_1 \frac{\partial}{\partial x_1} + x_2 \frac{\partial}{\partial x_2} - x_3 \frac{\partial}{\partial x_3} - x_4 \frac{\partial}{\partial x_4}.$
The corresponding resonance equation is: $\alpha_1 + \alpha_2 - \alpha_3 - \alpha_4 = 0$.  The algebra of algebraic first integrals
is generated by the functions $x_1x_3, x_1x_4, x_2,x_3, x_2x_4$; it has functional dimension 3 but cannot be generated by just 3 functions.

Thus, instead of specifying $q$ first integrals, 
from the geometrical point of view it is better to look at the whole algebra 
$\mathcal{F}$ of first integrals of an integrable system of type $(p,q)$.
Notice also that, if $f_{ij} \in \mathcal{F}$ ($i,j=1,\hdots,p$)  such that the matrix $(f_{ij})$ is invertible, then  by putting
\begin{equation}
 \hat{X_i} = \sum_{ij} f_{ij} X_j \ \ \text{for all} \ \ i=1,\hdots,p,
\end{equation}
we get another integrable system $(\hat{X_1},\hdots,\hat{X_p},F_1,\hdots,F_q)$, 
which, from the geometric point of view, is essentially the same as the original system.
So we have the following definition:

\begin{defn} \label{defn:GeometricEquivalence}
Two integrable dynamical  systems  $(X_1,\hdots,X_p,F_1,\hdots,F_q)$ and $(X'_1,\hdots,X'_p,F'_1,\hdots,F'_q)$
of type $(p,q)$ on a manifold $M$ are said to be {\bf geometrically equal}, if they have the same algebra of first 
integrals (i.e. $F'_1,\hdots,F'_p$ are functionally dependent of $F_1,\hdots,F_p$ and vice versa), and there exists
an invertible  matrix $(f_{ij})_{i=1,\hdots,p}^{j=1,\hdots,p}$ (i.e. whose determinant does not vanish anywhere), 
whose entries $f_{ij}$ are first integrals of the system, such that one can write 
\begin{equation}
X'_i = \sum_{j}f_{ij} X_j \ \ \forall \ i=1,\hdots,p. 
\end{equation}
Two integrable systems are said to be {\bf geometrically equivalent} if they become geometrically the same 
after a diffeomorphism. 
\end{defn}

\begin{remark}
Though the choice of first integrals is not important in Definition \ref{defn:GeometricEquivalence} 
of geometric  equivalence, the $q$-tuple $F_1,\hdots, F_q$ of first integrals in Definition \ref{defn:NondegenerateSingularPoint} of 
nondegeneracy must be chosen so that  not only they are functionally independent, 
but their homogeneous parts are also functionally independent.  
(According to a simple analogue of Ziglin's lemma \cite{Ziglin-Branching1982}, 
in the analytic case, such a choice is always possible).
\end{remark}

It's clear that, near a regular point, i.e. a point $z$ such that $X_1 \wedge \hdots \wedge X_p (z) \neq 0$, any
integrable system of  type $(p,q)$ will be locally geometrically equivalent to the rectified
system $(X_1 = \frac{\partial}{\partial x_1}, \hdots, X_p = \frac{\partial}{\partial x_p}, F_{1}= x_{p+1},\hdots, F_q = x_m)$. 
The question about the local structure of integrable systems becomes interesting only at singular points.

\subsubsection{Linearization and rigidity of nondegenerate singular points} 
The theorems presented in this subsubsection are from the paper \cite{Zung-Nondegenerate2012}.

Due to resonances, it is impossible to linearize integrable vector
fields in general (if there were no resonance relations, there would be no formal first integral either). 
So the following result about geometric lineariztion, i.e. linearization up to geometric equivalence, 
is the best that one can hope for in the case of analytic integrable systems:

\begin{theorem}
\label{thm:linearization-nondegenerate}
An analytic (real or complex) integrable system near a nondegenerate fixed point is locally geometrically equivalent to a nondegenerate 
linear integrable system, namely its linear part. 
\end{theorem}

\begin{proof}
The above theorem is a simple consequence of Theorem \ref{thm:PBmain2} about the existence of an effective analytic torus $\bbT^p$-action  
in the neighborhood of a nondegenerate fixed point of an integrable system of type $(p,q)$ (because the toric degree is equal to $p$ in this case), 
and the fact that nondegenerate linear integrable systems of type $(p,q)$ are essentially the same as linear $\bbT^p$-actions. 

Indeed, let $\sqrt{-1}Z_1,\hdots,\sqrt{-1}Z_p$ be the generators of the associated analytic torus action near nondegenerate a fixed point $O$
of an analytic integrable system $(X_1,\hdots,X_p, F_1\hdots, F_q).$ We can assume that the torus action is alreay linearied, i.e.,
$Z_1,\hdots, Z_p$ are linear vector fields. Since, for every i, $Z_i$ is tangent to the level sets of $(F_1,\hdots, F_q)$, and the tangent
space to these level sets at a generic point is spanned by $X_1,\hdots, X_p$, we have that
\begin{equation}
Z_i \wedge X_1 \wedge \hdots \wedge X_p = 0 
\end{equation}
for all $i=1,\hdots, p.$. Since $ Z_1,\hdots, Z_p$ are independent, by 
dimensional consideration, the inverse is also true: $X_i \wedge Z_1 \hdots \wedge Z_p = 0$ for all $i=1,\hdots, p.$ 
Lemma \ref{lemma:division} below says that we can write 
$X_i = \sum_{j} f_{ij} Z_j$ in a unique way, where $f_{ij}$ 
are local analytic functions, which are also first integrals of the system.
The fact that the matrix $(f_{ij})$ is invertible, i.e. it has non-zero determinant at $O$, 
is also clear, because $(Z_1, \hdots, Z_p)$
are nothing but a linear transformation of the linear part of $(X_1,\hdots,X_p).$

What we have proved is that,  near a nondegenerate fixed point, 
an integrable system is geometrically equivalent to its linear part, at least
in the complex analytic case. In the real analytic case, 
the vector fields $(Z_1,\hdots,Z_p)$ are not real in general, but the proof
will remain the same after a complexification, 
because the Poincaré-Dulac normalization in the real case can be chosen to be real.  
\end{proof}

\begin{lemma}[Division lemma] 
\label{lemma:division}
If  $(Y_1,\hdots,Y_p, G_1,\hdots, G_q)$ is a nondegenerate linear integrable system, and $X$ is a local analytic vector field 
which commutes with $Y_1,\hdots,Y_p$ and such that $X \wedge Y_1 \wedge \hdots \wedge Y_p = 0$, then we can write 
$X = \sum f_i Y_i$ in a unique way, where $f_i$ are local analytic functions which are common first integrals of $Y_1,\hdots, Y_p$.
\end{lemma}

\begin{proof}
 Without loss of generality, we may assume that $Y_i = \sum_{j} c_{ij} Z_j$, where $c_{ij}$ are integers and 
$Z_i = x_i \frac{\partial}{ \partial x_i}$ in some coordinate system $(x_1,\hdots, x_m)$. 
We will write $X = \sum_i g_i Z_i,$ where $x_i g_i$ are analytic functions.
The main point is to prove that $g_i$ are analytic functions, and the rest of the lemma will follow easily. Let $\prod_i x_i^{\alpha_i}$
be a polynomial first integral of the linear system. Then we also have $X(\prod_i x_i^{\alpha_i}) = 0,$ which implies that
$\sum_i \alpha_i g_i = 0.$ If $\alpha_1 \neq 1$ then  $x_1g_1 = (-\sum_{i=2}^m x_1g_i)/\alpha_1$ vanishes when
$x_1 = 0,$ and so $x_1g_1$ is divisible by $x_1$, which means that $g_1$ is analytic. Thus, for each $i$, if we can choose
a monomial first integral $\prod_i x_i^{\alpha_i}$ such that $\alpha_i \neq 0,$ then $g_i$ is analytic. Assume now that
all monomial first integrals $\prod_i x_i^{\alpha_i}$ must have $\alpha_1 = 0.$ It means that all the first integrals are also
invariant with respect to the vector field $Z_1 = x_1 \frac{\partial}{ \partial x_1}$. Then $Z_1$ must be a linear combination of
$Y_1,\hdots, Y_p$ (because the system is already ``complete'' and one cannot add another independent commuting vector field to it),
and we have $[Z_1, X] = 0.$ From this relation it follows easily that $g_1$ is also analytic in this case. Thus, all functions $g_i$
are analytic.
\end{proof}

Theorem \ref{thm:linearization-nondegenerate} can be extended to the case of non-fixed 
nondegenerate singular points in an obvious way, with the same proof, using 
the toric characterization of local normalizations of vector fields:
 
\begin{thm}\label{thm:linearization-nondegenerate2}
 Any analytic integrable dynamical system near a nondegenerate singular point is locally geometrically 
equivalent to a direct product of a linear nondegenerate integrable system and a constant (regular) integrable system.
\end{thm}

Another related result is the following deformation rigidity theorem for nondegenerate singular points:

\begin{theorem}\label{thm:rigid}
 Let 
%\begin{equation}
$(X_{1, \theta},\hdots, X_{p, \theta}, F_{1, \theta}, \hdots, F_{q, \theta})$ 
%\end{equation}
be an analytic family of integrable systems of type $(p,q)$
depending on a parameter $\theta$ which can be multi-dimensional: $\theta = (\theta_1,\hdots,\theta_s)$, 
and assume that $z_0$ is a nondegenerate fixed point when $\theta = 0$. Then there exists a local 
analytic family of fixed points $z_{\theta}$, such that $z_{\theta}$ is a fixed point of 
$(X_{1, \theta},\hdots, X_{p, \theta}, F_{1, \theta}, \hdots, F_{q, \theta})$ for each $\theta$, and moreover,
up to geometric equivalence, the local structure of  $(X_{1, \theta},\hdots, X_{p, \theta}, F_{1, \theta}, \hdots, F_{q, \theta})$ at $z_\theta$ 
does not depend on $\theta$.
\end{theorem}

\begin{proof}
We can put the integrable systems in this family together to get one ``big'' integrable system of type $(p, q+s)$, with the last
coordinates $x_{m+1}, \hdots, x_{m+s}$ as additional first integrals. Then $z_0$ is still a nondegenerate fixed point for this big integrable system,
and we can apply Theorem (\ref{thm:linearization-nondegenerate}) to get the desired result.   
\end{proof}

We also have an extension of Ito's theorem \cite{Ito-Birkhoff1989} to the non-Hamiltonian case. Ito's theorem says that, an analytic integrable 
Hamiltonian system at a non-resonant singular point (without the requirement of nondegeneracy of the momentum map at that point) can
also be locally geometrically linearized (i.e. locally one can choose the momentum map so that the system becomes nondegenerate and geometrically
linearizable). For Hamiltonian vector fields, there are many auto-resonances due to their Hamiltonian nature, which are not counted as
resonance in the Hamiltonian case. So, in the non-Hamiltonian case, we have to replace the adjective ``non-resonant'' by ``minimally-resonant'':

\begin{defn}
A vector field $X$ in a integrable dynamical system $(X_1 = X, \hdots, X_p, F_1,\hdots, F_q)$ of type $(p,q)$ 
is called {\bf minimally resonant} at a singular point $z$ if its toric degree at $z$ is equal to $p$ (maximal possible). 
 \end{defn}

\begin{theorem}  
Minimally-resonant singular points of analytic integrable systems are also 
locally geometrically linearizable in the sense that one can change the 
auxiliary commuting vector fields (keeping the first vector field and the functions intact) 
in order to obtain a new integrable system which is locally geometrically linearizable.
\end{theorem}

The proof of the above theorem is also similar to the proof of Theorem \ref{thm:linearization-nondegenerate} 
and is a direct consequence of the toric characterization of the Poincaré--Birkhoff normalization.

\subsubsection{Geometric linearization in the smooth case}

In the smooth case, we still have the same definitions of linear part, geometric equivalence, 
nondegeneracy and geometric linearization as in the analytic case. 
We have the following conjecture, which is the smooth version of Theorem \ref{thm:linearization-nondegenerate2}:

\begin{conjecture}
\label{conjecture:SmoothLinear}
 Any smooth integrable dynamical system near a nondegenerate singular point is locally geometrically 
smoothly equivalent to a direct product of a linear nondegenerate integrable system and a constant system.
\end{conjecture}

As of this writing, the above conjecture is still open in the general case.
The smooth case is much more complicated than the analytic case, because when the real toric degree is smaller than the 
toric degree, one cannot complexify a smooth system to find the torus action (whose dimension is equal to the toric
degree) in general. Nevertheless, we know that the conjecture is true in the following particular cases:

a) Hamiltonian systems. The smooth linearization theorem for non-degenerate singular points of smooth integrable Hamiltonian systems
was proved by Eliasson \cite{Eliasson-Thesis1984,Eliasson-Normal1990}. 
Strictly speaking Eliasson \cite{Eliasson-Thesis1984,Eliasson-Normal1990}
wrote down only a sketch of the proof in the case of non-elliptic singularities, though all the main ingredients are there. See 
also \cite{DufourMolino-AA,Zung-FocusII2002,Miranda-Thesis2003,Chaperon-Focus2013,SanWa-Focus2013} for details and
other methods of proof. For elliptic singularities of integrable Hamiltonian systems, one can also use the toric characterization
to prove the local geometric linearization theorem, like in the analytic case. (For hyperbolic singularities the situation is
more complicated, because one cannot complexify a smooth system in general in order to find a torus action).

b) Systems of type $(m,0)$, i.e. a family of $m$ independent commuting vector fields on a $m$-dimensional manifold. 
In this case, we have:

\begin{thm}\cite{Zung-Nondegenerate2012} 
\label{thm:NormalForm_n0}
Let $O$ be a nondegenerate singular point of rank $k$ 
a smooth integrable system $(X_1,\hdots, X_m)$  of type $(m,0)$. 
Then there exists a smooth local coordinate 
system $(x_1,x_2,..., x_m)$ in a neighborhood of $O$, non-negative integers $h, e \geq 0$ (which do not
depend on the choice of coordinates)
such that $h + 2e = m -k$, and a real invertible $n\times n$ matrix $(v_i^j)$ 
such that the vector fields $Y_i = \sum_{j=1}^n v_i^j X_j$ 
have the following form:
\begin{equation} 
\begin{cases}
Y_i = x_i\frac{\partial }{\partial x_i} \quad \forall \ i = 1,\hdots, h, \\
Y_{h+2j-1} = x_{h+2j-1}\frac{\partial }{\partial x_{h+2j-1}} +  x_{h+2j}\frac{\partial }{\partial x_{h+2j}},  \\
Y_{h+2j} = x_{h+2j-1}\frac{\partial }{\partial x_{h+2j}} -  
      x_{h+2j}\frac{\partial }{\partial x_{h+2j-1}} \quad \forall \ j = 1, \hdots, e, \\
Y_i = \frac{\partial }{\partial x_i} \quad \forall \ i = m-k+1, \hdots, m.
\end{cases}
\end{equation}
\end{thm}

One can prove the above theorem along the following arguments, which are due to 
a referee of the paper \cite{Zung-Nondegenerate2012}: 
In the case of a fixed point, the linear part of an appropriate  linear combination $E = \sum a_i X_i$
of the vector fields $X_1,\hdots, X_m$ is a radial vector field, i.e. has the form 
$E^{(1)} = \sum_{i=1}^n x_i \frac{\partial}{\partial x_i}$. By Sternberg's theorem, $E$ is smoothly
linearizable, i.e. we can assume that $E = \sum_{i=1}^m x_i \frac{\partial}{\partial x_i}$ after a smooth
change of the coordinate system. Since the vector fields $X_i$ commute with the radial vector field 
$E = \sum_{i=1}^m x_i \frac{\partial}{\partial x_i}$, they are automatically linear in the new coordinate
system. The case of a singular point of positive rank $k > 0$ can be reduced to the case of a fixed point, 
by considering the $m-k$-dimensional isotropy algebra of the infinitesimal $\mathbb{R}^m$-action 
generated by $X_1,\hdots,X_m$ at the singular point, and showing the existence of a
$m-k$-dimensional invariant submanifold of the subaction of this isotropy algebra, 
which is transverse to the local orbit through the singular point 
of the $\mathbb{R}^m$-action.

c) Systems of type $(1,m -1)$, i.e. a vector field with a complete set of first integrals. In this case we have:

\begin{thm}\cite{Zung-OrbitalSmooth2012} 
\label{thm:1_n-1_case}
Let $(X,F_1,\hdots,F_{n-1})$ be a smooth integrable system of type $(1, n-1)$ with a fixed point $O$ which satisfies the following
nondegeneracy conditions: \\
1)  The semisimple part of the linear part of $X$ at  $O$  is non-zero, and the $\infty$-jets
of $F_1,\hdots,F_{n-1}$ at $O$ are funtionally independent. \\
2)  If moreover $0$ is an eigenvalue  of $X$  at $O$ with multiplicity $k \geq 1$,
then the differentials of the functions $F_1,\hdots, F_k$ are linearly independent at $O$:
%\begin{equation}
 $dF_1(O) \wedge \hdots \wedge dF_k(O) \neq 0.$ \\
%\end{equation}
Then there exists a local smooth coordinate system $(x_1,\hdots,x_n)$ in which $X$ can be written as
\begin{equation}
X = F  X^{(1)},
\end{equation}
where $X^{(1)}$ is a semisimple linear vector field in $(x_1,\hdots,x_n)$, and $F$ is a smooth first
integral of $X^{(1)}$ such that $F(O) \neq 0.$
\end{thm}

(The condegeneracy condition 2 in the above theorem is conjectured to be superfluous).

d) Some low-dimensional cases. In particular, the case of 
non-Hamiltonian focus-focus singular points of smooth integrable systems
of type $(2,2)$ was studied by Jiang Kai (unpublished, talk presented in Barcelona in 09/2013).

Let us indicate here why we believe that the above conjecture is true, and some methods which could be used to prove it.

1) By geometric arguments similar to the ones used in \cite{Zung-Symplectic1996,Zung-Poincare2002,Zung-Birkhoff2005}, we can show
the existence of a smooth torus $\bbT^d$-action which preserves the system, where $d$ is the {\it real} toric degree
of the system. Up to geometric equivalence, 
we can also assume that the vector fields which generate this torus action are part of our system.
The remaining vector fields of the system are hyperbolic and invariant with respect to this smooth torus action.

2) Theorem \ref{thm:linearization-nondegenerate} is also true in the formal case with the same proof, 
and so we can apply a formal linearization to our smooth system. Together with Borel's theorem, 
it means that there is a local smooth coordinate system in which our system is already geometrically linear up to a flat term.

3) After the above Step 2, one can try to use  results and techniques on finite determinacy of mappings 
à la Mather \cite{Mather-Determinacy1969} to find  a matrix whose entries are smooth first integrals, such that 
when multiplying this matrix with our vector fields, we obtain a new geometrically equivalent system whose vectors
are linear + flat terms.

4) One can now try to invoke an equivariant version of Sternberg--Chen theorem \cite{Chen-Vector1963,Sternberg}, 
due to Belitskii and Kopanskii \cite{BK-Equivariant2002}, 
which says that smooth equivariant hyperbolic vector fields which are formally linearizable are also smoothly equivariantly linearizable.
Of course, we will have to do it simultaneously for all commuting hyperbolic vector fields. So we need an extension of 
the result of Belitskii and Kopanskii to the situation of a smooth 
$\bbR^k$-action with some hyperbolicity property which is formally linear. Maybe we would also need a version of 
Belitskii--Kopanskii--Sternberg--Chen for  vector fields which have first integrals. 

5) The results and techniques of Chaperon \cite{Chaperon-geometrie1986,Chaperon-RZ2013} for the smooth
linearization of $\bbR^k \times \bbZ^h$-actions  may be very useful here. Of course,
techniques from the integrable Hamiltonian case, e.g. \cite{ColinVey-Morse1979, DufourMolino-AA, Eliasson-Normal1990}, 
in particular division lemmas for nondegenerate smooth systems, can probably be extended to the non-Hamiltonian case as well.

6) Geometrically, at least at the linear level, via a spectral decomposition, 
a complicated nondegenerate singular point can be decomposed into a direct product of its {\it indecomposable components}. For example, in
the Hamiltonian case, there are only 3 kinds of  indecomposable nondegenerate singular components, 
namely 2-dimensional elliptic, 2-dimensional hyperbolic and 4-dimensional focus-focus
(see, e.g., \cite{Zung-Symplectic1996}). One can try to reduce the linearization problem for a complicated singular point
to the decomposition problem plus the linearization problem for each of its components. 

\subsection{Semi-local torus actions and normal forms} \hfill

Consider a level set 
\begin{equation}
N = \{F_1 = c_1, \hdots, F_q = c_q\} 
\end{equation}
of an integrable system $(X_1\hdots, X_p, F_1, \hdots, F_q)$ of type $(p,q)$ on a manifold $M$.
We will assume that $N$ is connected compact, but the Liouville's theorem fails for $N$, because
$N$ contains a singular point of the system. The set $N$ is partitioned into the orbits of the
$\bbR^p$-action generated by the commuting vector fields $X_1,\hdots,X_p$. Under some mild conditions,
$N$ will contain a compact orbit $O$ if this $\bbR^p$-action. A natural question arises: does there exist
a natural associated torus action near $N$ or near $O$, which preserves the system and which is transitive
on $O$? 

We know that the answer is yes, at least in the case of nondegenerate singularities. For example, 
in the case of integrable Hamiltonian systems, it has been shown in \cite{Zung-Symplectic1996} that if 
$N$ is a connected compact  nondegenerate singular level set of rank $k$ (i.e. $\dim O = k$ where
$O$ is a compact orbit of the system in $N$) then there exists a Hamiltonian torus $\bbT^k$-action
in an neigborhood of $N$ which preserves the system and which is transitive on $O$. 
A semi-local normal form (linearization) theorem near a compact nondegenerate singular orbit $O$ of
a smooth integrable Hamiltonian system was obtained by Miranda and Zung in \cite{MirandaZung-NF2004}, based
on this torus action and on the virtual commutativity of the automorphism group.
Miranda--Zung linearization theorem \cite{MirandaZung-NF2004} can probably be extended to the case
of nondegenerate compact singular orbits of integrable non-Hamiltonian systems 
(modulo Conjecture \ref{conjecture:SmoothLinear}).

The nondegeneracy condition is a natural condition and most singularities are nondegenerate. 
However, there are also degenerate singularities,
and one is also interested in the existence of associated torus actions for such singularities.
It turns out that, in the real analytic case, the so-called {\bf finite type condition}, 
which is much weaker than the
nondegeneracy condition, suffices. (In a ``reasonable'' system, all degenerate singularities will be of finite 
type). To formulate this condition, denote by $M_\bbC$ a small
open complexification of our manifold $M^m$ on which the complexification ${\bf
X}_\bbC, {\bf F}_\bbC$ of $\bf X$ and $\bf F$ exists, where $\bf X = (X_1,\hdots, X_p)$ denotes the $p$-tuple of
vector fields and ${\bf F} = (F_1,\hdots,F_q): M \to \bbR^q$ denote the $q$-vector valued first integral of our
integrable system. Denote by $N_\bbC$ a
connected component of ${\bf F}_\bbC^{-1}({\bf F}(O))$ which contains $N$.

\begin{defn} With the above notations, the singular orbit $O$ is called
of {\bf  finite type} if there is only a finite number of orbits of the
infinitesimal $\bbC^p$-action generated by ${\bf X}_\bbC$ in $N_\bbC$, and $N_\bbC$ contains a
regular point of the map ${\bf F}_\bbC$.
\end{defn}

\begin{thm}[\cite{Zung-Tedemule2003}]
\label{thm:tedemule-action} With the above notations, if $O$ is a compact
finite type singular orbit of dimension $r$, then there is a real analytic
torus action of $\bbT^r$ in a neighborhood of $O$ which preserves the
integrable system $({\bf X},{\bf F})$ and which is transitive on $O$. If
moreover $N$ is compact, then this torus action exists in a neighborhood
of $N$. 
%In the Hamiltonian case this torus action also preserves the
%Poisson structure.
\end{thm}

A very closely result to the above theorem and whose proof also uses the same ingredients is the
following theorem about the local automorphism group of an integrable systems near a compact singular 
orbit.  Denote by $\cA_O$ the local automorphism group of the integrable system
$({\bf X},{\bf F})$ at $O$, i.e. the group of germs of local analytic
diffeomorphisms in a neigborhood of $O$ which preserve ${\bf X}$ and ${\bf F}$. Denote by $\cA_O^0$ the
subgroup of $\cA_O$ consisting of elements of the type $g^1_Z$, where $Z$
is a analytic vector field in a neighborhood of $O$ which preserves the
system and $g^1_Z$ is the time-1 flow of $Z$. The torus in the previous
theorem is of course an Abelian subgroup of $\cA_O^0$. 

\begin{thm}[\cite{Zung-Tedemule2003}]
If $O$ is a compact finite type singular orbit as above,
then $\cA^0_O$ is an Abelian normal subgroup of
$\cA_O$, and $\cA_O/\cA^0_O$ is a finite group.
\end{thm}

Theorem \ref{thm:tedemule-action} reduces the study of the behavior of
integrable systems near compact singular orbits to the study of fixed
points with a finite Abelian group of symmetry (this group arises from the
fact that the torus action is not free in general, only locally free). For
example, as was shown in \cite{Zung-Degenerate2000}, the study of {\it
corank-1} singularities of Liouville-integrable systems is reduced to the
study of families of functions on a $2$-dimensional symplectic disk which
are invariant under the rotation action of a finite cyclic group
$\bbZ/\bbZ_k$, where one can apply the theory of singularities of
functions with an Abelian symmetry developed by Wassermann
\cite{Wassermann-Symmetry1988} and other people. A (partial)
classification up to diffeomorphisms of corank-1 degenerate singularities
was obtained by Kalashnikov \cite{Kalashnikov-1Corank1998} (see also
\cite{Zung-Degenerate2000,GoSt-1Corank1987}), and symplectic invariants
were obtained by Colin de Verdière \cite{Colin-Singular2003}.

Notice also that Theorem \ref{thm:tedemule-action}, together with Theorem \ref{thm:PBmain}
and the toric characterization of Poincaré-Birkhoff normalization,
provides an analytic Poincaré-Birkhoff normal form in the neighborhood a
singular invariant torus of an integrable system. More precisely, with the above notations,
we can formulate Theorem \ref{thm:PBorbit} below. First make a reduction (i.e. quotient)
of the system near $O$ with respect to the torus action in Theorem \ref{thm:tedemule-action}. Then
the vector fields of the reduced system vanishes at the image of $O$ under the reduction, and so we can
talk about the toric degree of this reduced system: it is equal to the toric degree of the reduction of 
$\sum_{i=1}^q a_i X_i$, where the numbers $a_i$ are in generic position. Denote this reduced toric number
by $t$.

\begin{thm} \label{thm:PBorbit}
With the above notations, in a neighborhood of $O$ in the complexified manifold $M_\bbC$ there is a natural 
effective analytic torus $\bbT^{t + \dim O}$-action which preserves the system $({\bf X}_\bbC, {\bf F}_\bbC)$,
and which leaves $O$ invariant and is transitive on $O$. The torus $\bbT^{\dim O}$-action in Theorem \ref{thm:tedemule-action} 
is a subaction of this torus action. Moreover, this torus action has the structure-preserving property: any tensor field preserved
by the system is also preserved by this torus action.
\end{thm}

The proof of the above theorem has not been written down explicitly anywhere, but is left to the reader as an exercise.

%%%%
\section{Geometry of integrable systems of type (n,0)}

In this section, we will study the geometry of smooth integrable systems of type $(n,0)$,
following a recent paper  with Nguyen Van Minh \cite{ZungMinh_Rn2014}. 
We refer to this paper for various details and proofs that will be omitted here.

Recall that, a smooth integrable system of type
$(n,0)$ means an $n$-tuple of commuting smooth vector fields $X_1,\hdots, X_n$
on a $n$-dimensional manifold $M^n$. (There is no function, just vector fields).
We will always assume in this section
that the system is $(X_1,\hdots, X_n)$ {\bf nodegenerate}, i.e. every singular point of it is nondegenerate.
Moreover, we will assume that commuting vector fields $X_1,\hdots, X_n$ are {\bf complete}, 
i.e. they generate an action of $\bbR^n$ on $M^n$, which we will denote by
\begin{equation}
\rho: \bbR^n \times M^n \to M^n.
\end{equation}

We will say that $\rho$ is a {\bf nondegenerate $\bbR^n$ action} on $M^n$, and that 
$X_1,\hdots, X_n$ are the {\bf generators} of $\rho$. Instead of takling about the system
$(X_1,\hdots, X_n)$, we will talk about the $\bbR^n$-action $\rho,$ which is the same thing.
If a point $z \in M^n$ is of rank $k$ with respect to the tystem 
$X_1,\hdots, X_n$, then it is also of rank $k$ with respect to $\rho,$ in the sense that the orbit
$\cO_z$ of $\rho$ through $z$ is of dimension $k$.

If $v = (v^1,\hdots,v^n)  \in \bbR^n$ is a non-trivial element of $\bbR^n$, then we put
\begin{equation}
 X_v = \sum_{i=1}^n v^i X_i
\end{equation}
and call it the {\bf generator of the action $\rho$ associated to $v$}. 
If $v_1,\hdots,v_n \in \bbR^n$ from a basis of $\bbR^n$, 
then the vector fields $X_{v_1},\hdots, X_{v_n}$ 
also generate the same $\bbR^n$ action as $\rho$, up to an automorphism of $\bbR^n$.

We will study the global geometry of nondegenerate $\bbR^n$-actions on $n$-manifolds. But first, we need some
local and semi-local normal form results.

\subsection{Normal forms and automorphism groups}
\label{subsection:singularities}

\subsubsection{Local normal forms and adapted bases}

Recall from Theorem \ref{thm:NormalForm_n0} that, if $z$ is a singular point of rank $k$ of a nondegenerate
smooth action $\rho: \bbR^n \times M^n \to M^n$ generated by commuting vector fields $X_1,\hdots, X_n,$ then there is
a smooth local coordinate system $(x_1,\hdots,x_n)$ in a neighborhood of $z$ and a basis  
$(v_1,\hdots,v_n)$ of $\bbR^n$ such that locally we have
\begin{equation}
\label{eqn:NormalForm_n0}
\begin{cases}
X_{v_i} = x_i\frac{\partial }{\partial x_i} \quad \forall \ i = 1,\hdots, h, \\
X_{v_{h+2j-1}} = x_{h+2j-1}\frac{\partial }{\partial x_{h+2j-1}} +  x_{h+2j}\frac{\partial }{\partial x_{h+2j}},  \\
X_{v_{h+2j}} = x_{h+2j-1}\frac{\partial }{\partial x_{h+2j}} -  
      x_{h+2j}\frac{\partial }{\partial x_{h+2j-1}} \quad \forall \ j = 1, \hdots, e, \\
X_{v_i} = \frac{\partial }{\partial x_i} \quad \forall \ i = n-k+1, \hdots, n,
\end{cases}
\end{equation}
where $X_{v_i} = \sum_{j=1}^n v_i^j X_j$ is the generator of $\rho$ associated to $v_i$ for each $i=1,\hdots,n.$

\begin{figure}[htb] 
\begin{center}
\includegraphics[width=115mm]{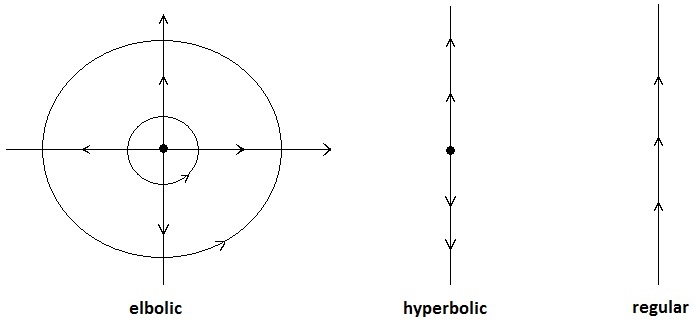}
\caption{Elbolic, hyperbolic, and regular components of $\bbR^n$-actions on $n$-manifolds.}
\label{fig:elbolic}
\end{center}
\end{figure}

The couple $(h,e)$ in the above formula does not depend on the choice of coordinates and bases, and is called the
{\bf HE-invariant} of the action $\rho$ at $z$.
The number $e$ is called the number of {\bf elbolic} components, and $h$ is called the number of
{\bf hyperbolic} components at $z$. The coordinate system $(x_1, \hdots, x_n)$ in the above formula
is called a local {\bf canonical system of coordinates}, and the basis $(v_1,\hdots, v_n)$ of $\bbR^n$
is called an {\bf adapted basis} of the action $\rho$ at $p$.
Local canonical coordinate systems at a point $p$ and associated adapted bases of $\bbR^n$
are not unique, but they are related to each other by the following theorem:

\begin{thm}\label{thm:CanonicalCoordinate}
Let $(x_1, \hdots, x_n)$ be a canonical system of coordinates at a point $z$ of a nondegenerate action $\rho$
together with an associated adapted basis $(v_1,\hdots, v_n)$ of $\bbR^n$.
Let $(y_1,\hdots, y_n)$ be another canonical system of coordinates at $z$ together with an associated adapted 
basis  $(w_1,\hdots, w_n)$ of $\bbR^n$. Then we have:

i) The vectors $(v_1,\hdots, v_h)$ are the same as the vectors $(w_1,\hdots, w_h)$ up to permutations, where 
$h$ is number of hyperbolic components.

ii) The $e$-tuples of pairs of vectors $((v_{h+1},v_{h+2}),\hdots, (v_{h+2e-1},v_{h+2e}))$ is also the 
same as the $e$-tuples $((w_{h+1},w_{h+2}),\hdots, (w_{h+2e-1},w_{h+2e}))$ up to permutations
and changes of sign of the type 
\begin{equation}
(v_{h+2i-1},v_{h+2i}) \mapsto (v_{h+2i-1},-v_{h+2i})
\end{equation}
(only the second vector, the one whose corresponding generator of $\rho$ is a vector field
whose flow is $2\pi$-periodic, changes sign).

iii) Conversely, if $(x_1, \hdots, x_n)$ and  $(v_1,\hdots, v_n)$ are as in 
Formula \eqref{eqn:NormalForm_n0}, 
and $(w_1,\hdots, w_n)$ is another basis of $\bbR^n$ which satisfies the above conditions i) and ii), then
$(w_1,\hdots, w_n)$ is the adapted basis of $\bbR^n$ for another canonical system of coordinates 
$(y_1,\hdots, y_n)$ at $z$.
\end{thm}

\begin{remark}
 The fact that the last vectors (from $w_{h+2e+1}$ to $w_n$) in an adapted basis can be arbitrary 
(provided that they form together with $w_1, \hdots, w_{h+2e}$ a basis of $\bbR^n$)
is very important in the global picture, because it allows us to glue different local canonical 
pieces together in a flexible way.
\end{remark}

It follows immediately from the local normal form formula \eqref{eqn:NormalForm_n0} that the singular set
$\cS = \{x \in M^n \ |\ \rank x < n\}$ 
of a nondegenerate action $\rho: \bbR^n \times M^n \to M^n$ is a stratified manifold, whose strata are 
$ \cS_{h,e} = \{x \in M^n \ |\ \text{HE-invariant of } x \text{ is } (h,e)\}$
given by the HE-invariant, and $\dim \cS_{h,e} = n -h-2e$ if $\cS_{h,e} \not= \emptyset$.
If $\cS \not= \emptyset$ then $\dim \cS  = n-1$ or $\dim \cS  = n-2$.
When there are singular points with hyperbolic components then $\dim \cS  = n-1$, and 
when there are only elbolic singularities ($h=0$ for every point) then  $\dim \cS  = n-2$. 

\begin{defn} \label{def:AssociatedVector}
 1) If $\cO_z$ is a singular orbit of corank 1 of a nondegenerate action $\rho: \bbR^n \times M^n \to M^n$,
i.e. the HE-invariant of $\cO_p$ is $(1,0)$, 
then the unique vector $v \in \bbR^n$ such that the corresponding generator $X_v$ of $\rho$ can be written as 
 $X_v = x\frac{\partial}{ \partial x}$ near each point of $\cO_z$ is called the {\bf associated vector} of $\cO_z$.

2) If $\cO_z$ is a singular orbit of HE-invariant $(0,1)$ (i.e. corank 2 transversally elbolic) then the couple of vectors 
$(v_1, \pm v_2)$ in $\bbR^n$, where $v_2$ is determined only up to a sign, such that $X_{v_1}$ and $X_{v_2}$ can be locally written as 
$$\begin{cases}
   X_{v_1} = x\frac{\partial}{ \partial x} + y\frac{\partial}{ \partial y}\\
 X_{v_2} = x\frac{\partial}{ \partial y} - y\frac{\partial}{ \partial x}
\end{cases}$$
is called the {\bf associated vector couple} of $\cO_z$.

\end{defn}

\subsubsection{Local automorphism groups and the reflection principle}

\begin{thm}\label{thm:LocalAutomorphism}
 Let $z$ be a nondegenerate singular point of HE-invariant $(h,e)$ and 
rank $r$ of an action $\rho: \bbR^n \times M^n \to M^n$ $(n= h + 2e + r)$. Then the group 
of germs of local isomorphisms (i.e. local diffeomorphisms which preserve the action) which fix the point $z$ 
is isomorphic to  $\bbT^e \times \bbR^{e+h} \times (\bbZ_2)^h$. The part $\bbT^e \times \bbR^{e+h}$
of this group comes from the action $\rho$ itself (internal automorphisms given by the action of the isotropy 
group of $\rho$ at $z$).
\end{thm}

The finite automorphism group $(\bbZ_2)^k$ in the above theorem acts not only locally in the neighborhood 
of a singular point $p$ of HE-invariant $(h,e)$, but also in the neighborhood of a smooth closed manifold 
of dimension $n-h-2e$ which contains $p$. More precisely, we have the following {\bf reflection principle}, 
which is somewhat similar to the Schwartz reflection principle in complex analysis:

\begin{thm}[Reflection principle] \label{thm:Reflection}
a) Let $z$ be a point of HE-invariant $(1,0)$ of a nondegenerate $\bbR^n$-action $\rho$ 
on a manifold $M^n$ without boundary. 
Denote by $v \in \bbR^n$ the associated vector of $z$ (i.e. of the orbit $\cO_z$) as in 
Definition \ref{def:AssociatedVector}. Put
\begin{equation} 
\cN_v = \{y \in M^n \ |\ X_v(y) = 0 \text{ and } X_v \text{ can be written as  } 
x_1\frac{\partial}{\partial x_1} \text{ near }  y\}.
\end{equation} 
Then $\cN_v$ is a smooth embedded hypersurface of dimension $n-1$ of $M^n$ 
(which is not necessarily connected), and there is a unique non-trivial involution 
$\sigma_v : \cU(\cN_v) \to \cU(\cN_v)$ from a neighborhood of $\cN_v$ to itself which preserves the 
action $\rho$ and which is identity on $\cN_v$.

b) If the HE-invariant of $z$ is $(h,0)$ with $h>1$, then
we can write 
\begin{equation} 
z \in \cN_{v_1,\hdots,v_h} = \cN_{v_1}\cap \hdots \cap \cN_{v_h}
\end{equation} 
where $\cN_{v_i}$ are defined as in a), $(v_1,\hdots,v_h)$ is a free family of vectors in $\bbR^n$,
the intersection $\cN_{v_1}\cap \hdots \cap \cN_{v_h}$ is transversal and $\cN_{v_1,\hdots,v_h}$
is a closed smooth submanifold of codimension $h$ in $M$. The involutions $\sigma_{v_1},\hdots, \sigma_{v_h} $
generate a group of automorphisms of $(\cU(\cN_{v_1,\hdots,v_h}), \rho)$ isomorphic to $(\bbZ_2)^h$.
\end{thm}

\subsubsection{Semi-local norml forms}

Consider an orbit $\cO_z = \{ \rho(t,z) \ |\ \ t \in \bbR^n \}$ though a point $z \in M^n$
of a given $\bbR^n$-action $\rho$. Since $\cO_z$ is a quotient of $\bbR^n$, 
it is diffeomorphic to $\bbR^k \times \bbT^l$ for some
nonnegative integers $k,l \in \bbZ_+$. 

\begin{defn} \label{defn:HERT-invariant}
The {\bf HERT-invariant} of an orbit $\cO_z$ or a point $z$ in it is
the quadruple $(h,e,r,t)$, where $h$ is the number of transversal hyperbolic components, 
$e$ is the number of transversal elbolic components, 
and $\bbR^r \times \bbT^t$ is the diffeomorphism type of the orbit.
\end{defn}

An orbit is compact if and only if $r = 0$, in which case it is a torus of dimension $t$.
We have the following linear model for a tubular neighborhood of a compact orbit of 
HERT-invariant $(h,e,0,t)$:

\begin{itemize}
\item The orbit is
\begin{equation}
\{0\} \times \{0\} \times \bbT^t / (\bbZ_2)^k,
\end{equation}
which lies  in 
\begin{equation}
B^{h} \times B^{2e} \times \bbT^t / (\bbZ_2)^k,
\end{equation}
(where $B^h$ is a ball of dimension $h$), with coordinates $(x_1, \hdots,x_{h+2e})$ 
on $B^{h} \times B^{2e}$ and  $(z_1, \hdots,z_t) \mod 2\pi$ 
on $\bbT^t$, and $k$ is some nonnegative integer such that $k\leq  \min(h,t)$.
\item The (infinitesimal) action of $\bbR^n$   is generated by the vector fields 
\begin{equation}
\begin{cases}
Y_i = x_i\frac{\partial }{\partial x_i} \quad \forall \quad i = 1, \hdots, h \\
Y_{h+2j-1} = x_{h+2j-1}\frac{\partial }{\partial x_{h+2j-1}} +  x_{h+2j}\frac{\partial }{\partial x_{h+2j}}  \\
Y_{h+2j} = x_{h+2j-1}\frac{\partial }{\partial x_{h+2j}} -  x_{h+2j}\frac{\partial }{\partial x_{h+2j-1}} \quad \forall \quad j = 1, \hdots, e \\
Y_{h+2e+i} = \frac{\partial }{\partial z_i} \quad \forall \quad i = 1, \hdots, t.
\end{cases}
\end{equation}
like in the local normal form theorem.
  \item The Abelian group $(\bbZ_2)^k$ acts on $B^h \times B^{2e} \times \bbT^t $ freely,
 component-wise, and by isomorphisms of the action, so that the quotient is still a manifold with
 an induced action of $\bbR^n$ on it. The action of $(\bbZ_2)^k$ on $B^{h}$ is 
by an injection from $(\bbZ_2)^k$ to the involution group $(\bbZ_2)^h$ generated by the
reflections $\sigma_i: (x_1,\hdots,x_i,\hdots, x_h) \mapsto (x_1,\hdots, -x_i,\hdots, x_h)$, its
action on $B^{2e}$ is trivial, and its action on $\bbT^t$ is via an injection of $(\bbZ_2)^k$
into the group of translations on $\bbT^t$.
\end{itemize}

\begin{thm} \label{thm:semi-localform}
Any compact orbit of a nondegenerate action
$\rho: \bbR^n \times M^n \to M^n$ can be linearized, i.e. there is a tubular neighborhood of it
which is, together with the action $\rho$, isomorphic to a linear model described above.
\end{thm}

More generally, for any point $z$ lying in a orbit $\cO_z$ of HERT-invariant $(h,e,r,t)$ which is not necessarily
compact (i.e. the number $r$ may be strictly positive), we still have the following linear model:

\begin{itemize}
\item The intersection of the orbit with the manifold is
\begin{equation}
 \{0\} \times \{0\} \times \bbT^t / (\bbZ_2)^k \times B^r,
 \end{equation}
which lies in 
\begin{equation}
(B^{h} \times B^{2e} \times \bbT^t / (\bbZ_2)^k) \times B^r
\end{equation}
 with coordinates $(x_1, \hdots,x_{h+2e})$ on $B^{h} \times B^{2e}$, 
$(z_1, \hdots,z_t) \mod 2\pi$ on $\bbT^t$, and $\zeta_1,\hdots, \zeta_r$ on $B^r$
 and $k$ is some nonnegative integer such that $k\leq  \min(h,t)$.
\item The action of $\bbR^n$   is generated by the vector fields 
\begin{equation}
\begin{cases}
Y_i = x_i\frac{\partial }{\partial x_i} \quad \forall \quad i = 1, \hdots, h, \\
Y_{h+2j-1} = x_{h+2j-1}\frac{\partial }{\partial x_{h+2j-1}} +  x_{h+2j}\frac{\partial }{\partial x_{h+2j}}  \\
Y_{h+2j} = x_{h+2j-1}\frac{\partial }{\partial x_{h+2j}} -  x_{h+2j}\frac{\partial }{\partial x_{h+2j-1}} \quad \forall \quad j = 1, \hdots, e, \\
Y_{h+2e+i} = \frac{\partial }{\partial z_i} \quad \forall \quad i = 1, \hdots, t, \\
Y_{h+2e+t+i} = \frac{\partial }{\partial \zeta_i} \quad \forall \quad i = 1, \hdots, r.
\end{cases}
\end{equation}
  \item The Abelian group $(\bbZ_2)^k$ acts on $\bbR^h \times \bbR^{2e} \times \bbT^t $ freely
in the same way as in the case of a compact orbit.
\end{itemize}

\begin{thm} \label{thm:semi-localform2}
Any point $q$ of any HERT-invariant $(h,e,r,t)$ with respect to a nondegenerate action 
$\rho: \bbR^n \times M^n \to M^n$ admits a neighborhood which is isomorphic to a linear
model described above.
\end{thm}

Theorem \ref{thm:semi-localform2} it simply a parametrized version of Theorem \ref{thm:semi-localform}, and can
also be seen as a corollary of Theorem \ref{thm:semi-localform}

\begin{remark}
 The difference between the compact case and the noncompact case is that, when $\cO_q$ is a compact orbit,
we have a linear model for a whole tubular neighborhood of it, but when $\cO_q$ is noncompact we have
a linear model only for a neighborhood of a ``stripe'' in $\cO_q$.
\end{remark}

\subsubsection{The twisting groups}

The minimal required group $(\bbZ_2)^k$ in Theorem \ref{thm:semi-localform}
and Theorem \ref{thm:semi-localform2}  is naturally isomorphic to the group
\begin{equation} \label{eqn:TwistingGroup}
 G_q = (Z_\rho (q) \cap (Z_\rho \otimes  \bbR))/Z_\rho.
\end{equation}

\begin{defn} \label{defn:TwistingGroup}
The group $G_q$ defined by the above formula is called the {\bf twisting group} of the action $\rho$ at $q$ (or at
the orbit $\cO_q$). The orbit $\cO_q$ is said to be {\bf non-twisted} (and $\rho$ is said to be non-twisted at $q$) 
if $G_q$ is trivial, otherwise it is said to be {\bf twisted}.
\end{defn}

\begin{remark}
The twisting phenomenon also appears in real-world physical integrable  Hamiltonian systems, 
and it was observed, for example, by Fomenko and his collaborators in their study of 
integrable Hamiltonian systems with 2 degrees
of freedom. See, e.g., \cite{BolsinovFomenko-IntegrableBook}. 
\end{remark}

\subsection{Induced torus action and reduction}

\subsubsection{The toric degree} 

Given a nondegenerate action $\rho: \bbR^n \times M^n \to M^n$, denote by
\begin{equation}
Z_{\rho} = \{ g \in \bbR^n : \rho(g,.) = Id_{M^n}\}
\end{equation}
the isotropy group of $\rho$ on $M^n$. Since $\rho$ is locally free almost everywhere due to 
its nondegeneracy,
$Z_{\rho}$ is a discrete subgroup of $\bbR^n$, so we have 
%\begin{equation}
$Z_{\rho} \cong  \bbZ^k$
%\end{equation}
for some integer $k$ such that $0 \leq k \leq n.$  The action $\rho$ of $\bbR^n$ descends to an action of 
%\begin{equation}
$\bbR^n / Z_{\rho} \cong \bbT^k \times \bbR^{n-k}$
%\end{equation}
on $M$, which we will also denote by $\rho$:
\begin{equation}
\rho : (\bbR^n/Z_\rho) \times M^n \to M^n.
\end{equation}

We will denote by 
\begin{equation} \label{eqn:rhoT-action}
\rho_\bbT: \bbT^k \times M^n \to M^n
\end{equation}
the subaction of $\rho$ given by the subgroup $\bbT^k \subset \bbT^k \times \bbR^{n-k} \cong \bbR^n/Z_\rho$. 
More precisely, $\rho_\bbT$ is the action of $(Z_{\rho}\otimes \bbR)/Z_{\rho}$ on $M^n$ induced from $\rho$,  which becomes a $\bbT^k$-action after an isomorphism from $(Z_{\rho}\otimes \bbR)/Z_{\rho}$ to $\bbT^k$. We will call $\rho_\bbT$ the {\bf induced torus action} of $\rho.$

\begin{defn}
The number $k = \rank_\bbZ Z_{\rho}$ is called the {\bf toric degree} of the action $\rho$. 
If the toric degree is equal to 0 then $\rho$ is called a {\bf totally hyperbolic action}.
\end{defn}

\begin{remark}
If $M^n$ admits an $\bbR^n$-action of toric degree $k$, then in particular it must admit an 
effective $\bbT^k$-action. When $k \geq 1$, this condition is a strong 
topological condition. For example, Fintushel \cite{Fintushel-Circle1977} 
showed (modulo Poincaré's conjecture which is now a theorem) that among
simply-connected 4 manifolds, only the manifolds 
$\bbS^4, \bbC \bbP^2,-\bbC \bbP^2,\bbS^2 \times \bbS^2$ and their connected sums
admit an effective locally smooth $\bbT^1$-action. 
This list is the same as the list of simply-connected 4-manifolds admitting an effective $\bbT^2$-action,
according to Orlik and Raymond \cite{OR_Torus1}, \cite{OR_Torus2}. A classification of non-simply-connected
4-manifolds admitting an effective $\bbT^2$-action can be found in Pao \cite{Pao-TorusAction1}.
\end{remark}

Even though the toric degree is a global invariant of the action, it can in fact be determined semi-locally 
from the HERT-invariant of any point on $M$ with respect to the action. More precisely, we have:

\begin{thm} \label{thm:HERT-toricdegree}
Let $\rho: \bbR^n \times M^n \to M^n$ be a nondegenerate smooth action of $\bbR^n$ on a $n$-dimensional 
manifold $M^n$ and $p \in M$ be an arbitrary point of $M$. If the HERT-invariant of $p$ with respect to $\rho$
is $(h,e,r,t)$, then the toric degree of $\rho$ on $M$ is equal to $e + t$.
\end{thm}

\begin{proof}
See \cite{ZungMinh_Rn2014} for the proof. It consists of the following five steps, and each step is based on
relatively simple topological arguments:

\underline{Step 1}: If $z \in M$ is a regular point then $\text{toric degree}(\rho) \leq t(z)$.

\underline{Step 2}: If $\cO_1$ and $\cO_2$ are two arbitrary different 
regular orbits then $Z_\rho (\cO_1) = Z_\rho (\cO_2),$ where $Z_\rho(\cO) \subset \bbR^n$
denotes the isotropy group of $\rho$ on $\cO$.

\underline{Step 3}: $Z_\rho = Z_\rho(\cO)$ for any regular orbit $\cO$. In particular, for any regular point $z$, 
 the toric rank of $\rho$ is equal to $t(z)$, and $e(z) = h(z) = 0$. 

\underline{Step 4}: If $z \in M^n$ is a singular point then $e(z) + t(z) \geq $ toric degree $(\rho)$.

\underline{Step 5}: The converse is also true: $e(z) + t(z) \leq $ toric degree $(\rho)$.
\end{proof}

The simplest case of nondegenerate systems of type $(n,0)$ is when the toric degree of $\rho$ is equal to $n$.
This case is a special case of Liouville's theorem: we have an effective action of $\bbT^n$ on $M^n$, 
and $M^n$ itself is diffeomorphism to the torus $\bbT^n$.

\subsubsection{Quotient space and reduced action} 
In general, if the toric degree $t(\rho)$ is positive, i.e. if the action 
$\rho: \bbR^n \times M^n \to M^n$ is not totally hyperbolic, then it
naturally projects down to an action of
$\bbR^n / (Z_\rho \otimes \bbR) \cong \bbR^{r(\rho)}$  on the quotient space $Q = M^n/\rho_\bbT$ of $M^n$ 
by the induced torus action $\rho_\bbT$, which we will denote by $\rho_\bbR:$
\begin{equation}
 \rho_\bbR: \bbR^{r(\rho)} \times Q \to Q,
\end{equation}
after an identification of $\bbR^n / (Z_\rho \otimes \bbR)$ with $\bbR^{r(\rho)}$.
Here $r(\rho) = \dim \bbR^n/(Z_\rho \otimes \bbR) = n - t(\rho) = \dim Q.$ 
We will call $\rho_\bbR$ the 
{\bf reduced action} of $\rho$.

There is a small technical problem. Namely, due to the singularities and the twisting groups, in general the quotient space $Q$ is not
a manifold  but an orbifold with boundary and corners. More precisely, 
it follows from the  normal form theorems that for every point $z \in Q$, 
locally a neigborhood $z$ in $Q$ is diffeomorphic to a direct product of the type
\begin{equation}
(D^2_1/\bbT^1_1) \times \hdots \times (D^2_e/\bbT^1_e) \times (B^h / G_z) \times B^r,
\end{equation}
where $(h,e,r,t)$ is the HERT invariant and $G_z$ is the twisting group of $z$ (i.e. of any point in $M^n$ whose image under 
the projection $M^n \to Q$ is $z$), $B^r$ and $B^h$ are balls of dimensions $r$ and $h$ respectively,
and each $D^2_i/\bbT^1_i$ is a half-closed interval obtained as the quotient of a 2-dimensional disk $D^2_i$ by 
the standard rotational action of $SO(2) \cong \bbT^1$.

Due to this fact, we have to extend the notion of nondegenerate 
$\bbR^r$-actions to the case of orbifolds: it simply means that, locally, we have a nondegenerate (infinitesimal) 
$\bbR^r$-action on a local branched covering space, which is a manifold together with a finite group action on 
it so that the quotient by that finite group action is our local orbifold, and we require that the $\bbR^r$-action 
commutes with the finite group action so that it can be projected to an $\bbR^r$-action on the orbifold.
In the case with boundary and corners, the boundary components are singular orbits of the action. 
The notions of toric degree 
can be naturally be extended to the case of actions on orbifolds with boundary and corners, 
and if the toric degree is 0 we will also say that the action is totally hyperbolic. With this in mind, 
we have the following reduction theorem:

\begin{thm} \label{thm:hypAction-quotientSpace}
Let $\rho: \bbR^n \times M^n \to M^n$ be a nondegenerate action of toric degree $t(\rho)$
on a connected manifold $M^n$, and put $r = r(\rho) = n - t(\rho)$. Then the quotient
space $Q = M^n/\rho_\bbT$ of $M^n$ by the associated  torus action $\rho_\bbT$ is an
orbifold of dimension $r$, and the reduced action $\rho_\bbR$ of $\bbR^n/(Z_\rho \otimes \bbR) \cong \bbR^r$
on $Q$ is totally hyperbolic.  If
the twisting group $G_z$ is trivial for every point $z \in M^n$, then $Q$ is a manifold with boundary and corners.
\end{thm}

\subsubsection{Cross multi-sections and reconstruction}

In the case when $(M^n, \rho)$ has no twistings, $Q$ is a manifold with boundary and corners,
and one can talk about cross sections of the singular torus fibration $M^n \overset{\bbT^{t(\rho)}}{\longrightarrow} Q$
over $Q$. We will say that an embeded submanifold with boundary and corners $Q_c \subset M^n$ is a smooth {\bf cross section} of the 
singular fibration $M^n \to Q$ if the projection map $proj. : Q_c \to Q$ is a diffeomorphism. 
The existence of a cross section
is equivalent to the fact that the \emph{desingularization via blowing up} 
of $M^n \to Q$ is a trivial principal $\bbT^{t(\rho)}$-bundle.
(The blowing up process here does not change the quotient space of the action $\rho_\bbT$ on $M^n$, but changes every singular orbit 
of $\rho_\bbT$ into a regular orbit, and changes $M^n$ into a manifold with boundary and corners, see Figure \ref{desingularization} 
for an illustration. This blow-up process is a standard one, and it was used for example by Dufour and Molino \cite{DufourMolino-AA} 
in the construction
of action-angle variables near elliptic  singularities of integrable Hamiltonian systems).

\begin{figure}[htb] 
\begin{center}
\includegraphics[width=80mm]{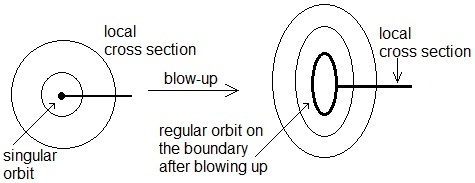}
\caption{Desingularization of $M^n \to Q$ by blowing up.}
\label{desingularization}
\end{center}
\end{figure}

In the case when $(M^n, \rho)$ has twistings, a priori $Q$ is only an orbifold and we cannot 
have a submanifold $Q_c$ in $M^n$ diffeomorphic to $Q$. In this case, instead of sections, one can talk about  
multi-sections: a smooth {\bf multi-section} of $M^n \to Q$ is a smooth embedded 
submanifold with boundary and corners $Q_c$ in $M^n$, together with a finite subgroup $G \subset(Z_\rho \otimes \bbR)/Z_\rho$
such that $Q_c$ is invariant with respect to $G$ (i.e. if $z\in Q_c$ and 
$w \in G$ then $\rho(w,z) \in Q_c$), and $Q_c/G \cong Q$ via the projection.
Remark that multi-sections also appear in many other places in the literature. 
For example, Davis and Januskiewicz in \cite{Davis-Convex1991} used them in their study of quasi-toric manifolds. 
They were also used in \cite{Zung-Symplectic1996}
in the construction of partial action-angle coordinates near singular fibers of integrable Hamiltonian systems.

\begin{prop}
\label{prop:multisection} i) If $(M^n, \rho)$ has no twistings, 
then the singular torus fibration $M^n \to M^n/\rho_\bbT = Q$ 
admits a smooth cross section $Q_c$.

ii) If $(M^n,\rho)$ has twistings, then the singular torus fibration $M^n \to M^n/\rho_\bbT = Q$
admits a smooth multi-section $(Q_c,G)$, where $G \subset(Z_\rho \otimes \bbR)/Z_\rho $ is generated by the 
twisting groups $G_z$ ($z \in M$) of $(M^n,\rho)$.
\end{prop}

The first part of above proposition can be proved using sheaf theory, based on the existence of local cross sections and
the contractibility of $Q$, 
%(which implies the vanishing of a corresponding Cech cohomology group), 
which will be explained in Subsection \ref{subsection:hyperbolic}. The second part follows from the first part
and an appropriate covering of $(M^n,\rho)$.

The cross (multi-)sections allow one to go back (i.e. reconstruct) from $(Q, \rho_\bbR)$ to $(M,\rho).$
In particular, we have:

\begin{thm} \label{cor:EquivariantUniqueness}
 Assume that $(M_1^n,\rho_1)$ and $(M_2^n,\rho_2)$ have the same quotient space $M^n_1/\rho_{1 \bbT} = M_2^n/ \rho_{2 \bbT} = Q$,
and moreover they have the same isotropy  at every point of $Q$: $Z_{\rho_1}(z) = Z_{\rho_2}(z)\ \forall\ z \in Q$, where
$Z_{\rho_1}(z)$ means the isotropy group of $\rho_1$ on the $\rho_{1 \bbT}$-orbit corresponding to $z$. Then $\rho_{1 \bbT}$ 
and $\rho_{2 \bbT}$ are isomorphic, i.e. there is a
diffeomorphism $\Phi: M_1^n \to M_2^n$ which sends $\rho_{1 \bbT}$ to $\rho_{2 \bbT}$.
\end{thm}

\begin{proof}
 Simply send a multisection in $M_1^n$ over $Q$ to a multi-section 
in $M_2^n$ over $Q$ by a diffeomorphism which projects to the identity
map on $Q$, and extend this diffeomorphism to the whole $M_1^n$ in the unique equivariant way with respect to the associated
torus actions. The fact that the isotropy groups are the same allows us to do so.
\end{proof}

Beware that, even though the two torus actions $\rho_{1 \bbT}$ and $\rho_{2 \bbT}$ in the above theorem are isomorphic,
and even if we assume that the two reduced actions $\rho_{1 \bbR}$ and $\rho_{2 \bbR}$ on $Q$ are the same, it does not mean that
$\rho_1$ and $\rho_2$ are isomorphic. The difference between the isomorphism classes of $\rho_1$ and $\rho_2$ can be measured
in terms of an invariant called the \emph{monodromy}, which will be explained in Subsection \ref{subsection:monodromy}.

\subsection{Systems  of toric degree n-1 and n-2}

\subsubsection{The case of toric degree n-1}

Consider  a nondegenerate action $\rho$ of toric degree $n-1$ on a compact
connected manifold $M^n$, an orbit $\cO_p$ of this action, and denote by $(h,e,r,t) $ the HERT-invariant of $\cO_p$.

According to Theorem \ref{thm:HERT-toricdegree}, we have $e+t = n-1$. On the 
other hand, the total dimension is $n = h + 2e + r + t$. These two equalities imply that 
$h+e + r= 1$, which means that one of the three numbers $h, e, r$ is equal to 1 and the other two 
numbers are 0. So we have only three possibilities: 

1) $r = 1, h = e = 0, t = n-1$, and
$\cO_p \cong \bbT^{n-1} \times \bbR$ is a regular orbit. The action $\rho_\bbT$
of $\bbT^{n-1}$ on such an orbit is free with the orbit space diffeomorphic to an open
interval.

2) $r = e = 0, h = 1, t = n-1$, and $\cO_p \cong \bbT^{n-1}$ is a compact singular
orbit of codimension 1 which is transversally hyperbolic. The action $\rho_\bbT$
of $\bbT^{n-1}$ on such an orbit is locally free; it is either free (the non-twisted case)
or have the isotropy group equal to $\bbZ_2$ (the twisted case).

3) $e = 1, h = r = 0, t = n-2$, and
$\cO_p \cong \bbT^{n-2}$ is a compact singular
orbit of codimension 2 which is transversally elbolic.

The orbit space $Q = M^n/\bbT^{n-1}$ of the action 
\begin{equation}
\rho_\bbT: \bbT^{n-1} \times M^n \to M^n
\end{equation}
is a compact one-dimensional manifold with or without
boundary, on which we have the reduced $\bbR$-action $\rho_\bbR$. 
The singular points of this $\bbR$-action
on $M^n/\bbT^{n-1}$ correspond to the singular orbits  of $\rho$. Since the toric degree is $n-1$ and not
$n$ and $M$ is compact, $\rho$ must have at least one singular orbit, and hence the quotient space
$Q = M^n/\bbT^{n-1}$ contains at least one singular point.
Topologically, $Q$ must be a closed interval or a circle, and globally, we have the following 4 cases:

\begin{figure}[htb]
\begin{center}
\includegraphics[width=100mm]{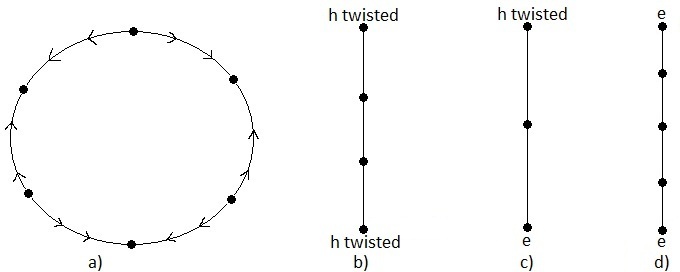}
\caption{The 4 cases of toric degree $n-1$}
\label{fig:n-1}
\end{center}
\end{figure}

\underline{Case a}: \emph{$Q$ is a circle, which contains  $m > 0$  hyperbolic points with respect to $\rho_\bbR$}. 
Notice that, $m$ is necessarily an even number, because the vector field which generates
the hyperbolic $\bbR$-action $\rho_\bbR$ on $Q$ changes direction on adjacent regular intervals, see Figure \ref{fig:n-1}a
for an illustration. The $\bbT^{(n-1)}$-action is free in this case, so $M^n$ is a $\bbT^{n-1}$-principal bundle 
over $Q$. Any homogeneous $\bbT^{n-1}$-principal bundle over a circle is trivial, so $M^n$ is diffeomorphic to
$\bbT^n \cong \bbT^{n-1} \times \bbS^1$ in this case.

\underline{Case d}: \emph{$Q$ is an interval, and each endpoint of $Q$ corresponds to a transversally
elbolic orbit of $\rho$}.
Topologically, in this case, the manifold $M^n$ can be obtain by gluing 2 copies of the ``solid torus'' $D^2 \times \bbT^{n-2}$
together along the boundary. When $n=2$, there is only one way to do it, and $M$ is diffeomorphic to a sphere $\bbS^2$.
When $n \geq 3 $, the gluing can be classified by the homotopy class (up to conjugations) of the two vanishing 
cycles on the common boundary $\bbT^{n-1}$. 
When $n=3$, the manifold $M^3$ is either $\bbS^2 \times \bbS^1$
(if the two vanishing cycles are equal up to a sign) or a 3-dimensional \emph{lens space}.
  
\underline{Case c}: \emph{$Q$ is an interval, one endpoint of $Q$ corresponds to a twisted transversally
hyperbolic orbit of $\rho$, and the other endpoint corresponds to a transversally elbolic orbit of $\rho$}.
Due to the twisting, the ambient manifold is non-orientable in this case. But $(M^n,\rho)$
admits a double covering $\widetilde{(M^n,\rho)}$ which belongs to Case b. 
If $n=2$ then $M^2 = \bbR \bbP^2$ in this case.
 
\underline{Case b}: \emph{$Q$ is an interval, and each endpoint of $Q$ corresponds to a twisted transversally
hyperbolic orbit of $\rho$}. Again, in this case, $M$ is non-orientable, but $(M^n,\rho)$ 
admits a normal $(\bbZ_2)^2$-covering $\widetilde{(M^n,\rho)}$ 
which is orientable and belongs to Case a. If $n=2$ then $M^2$ is a Klein bottle in this case.
 
We can classify actions of toric degree $n-1$ on closed manifolds as follows:

View $Q$ as a non-oriented graph, with singular points as vertices.
Mark each vertex of $Q$ with the vector or the vector couple in $\bbR^n$ associated to the corresponding orbit
of $\rho$ (in the sense of Definition \ref{def:AssociatedVector}). 
Then $Q$ becomes a marked graph,  which will be denoted by $Q_{\text{marked}}$. 
Note that $Q_{\text{marked}}$ and the isotropy group $Z_\rho \subseteq \bbR^n$ are invariants of $\rho$, 
which satisfy the following conditions ($C_i$)-($C_{iv}$):

\begin{itemize}
 \item[$C_i$)] $Q$ is homeomorphic to a circle or an interval.
 If $Q$ is a circle then it has an even positive number of vertices. 
 \item[$C_{ii}$)] Each interior vertex of $Q$ is marked with a vector in $\bbR^n$. If $Q$ is an interval then 
each end vertex of $Q$ is marked with either a vector or a couple of vectors of the type 
$(v_1, \pm v_2)$ in $\bbR^n$ (the second vector in the couple is only defined up to a sign).
 \item[$C_{iii}$)] $Z_\rho$ is a lattice of rank $n-1$ in $\bbR^n$. 
 \item[$C_{iv}$)] If $v \in \bbR^n$ is the mark at a vertex of $Q$, then
\begin{equation}
\bbR.v \oplus (Z_\rho \otimes \bbR) = \bbR^n.
\end{equation}
 If $(v, \pm w)$ is the mark at a vertex of $Q$, then we also have
\begin{equation}
\bbR.v \oplus (Z_\rho \otimes \bbR) = \bbR^n
\end{equation}
while $w$ is a primitive element of $Z_\rho$. Moreover, if $v_i$ and $v_{i+1}$ are two consecutive marks
(each of them may belong to a couple, e.g. $(v_i, \pm w_i)$), then they lie on different 
sides of $Z_\rho\otimes \bbR$ in $\bbR^n$.
\end{itemize}

In the case when $Q$ is a  circle, then there is another invariant of $\rho$, 
called the \emph{monodromy} and defined as follows:

Denote by $F_1, \hdots, F_m$ the $(n-1)$-dimensional orbits of $(M^n, \rho)$ in cyclical order (they correspond to
vertices of $Q$ in cyclical order). Denote by $\sigma_i$ the reflection associated to $F_i$. Let $z_1 \in M$ be an arbitrary 
regular point which projects to a point lying between the images of $F_m$ and $F_1$ in $Q$.
Put $z_2 = \sigma_1(z_1)$ (which is a point lying on the regular orbit between $F_1$ and $F_2$), 
$z_3 = \sigma_2(z_2),\hdots,z_{m+1} = \sigma_m(z_m)$. Then $z_{m+1}$ lies on the same regular orbit as $z_1$,
and so there is a unique element $\mu \in \bbR^n/Z_\rho$ such that 
\begin{equation}
z_{m+1} = \rho(\mu, z_1).
\end{equation}
This element $\mu$ is called the {\bf monodromy} of the action. Notice that $\mu$ does not depend on the choice of $z_1$
nor on the choice of $F_1$ (i.e. which singular orbit is indexed as the first one), but only on the choice of the orientation
of the cyclic order on $Q$: If we change the orientation 
of $Q$ then $\mu$ will be changed to $-\mu$. So a more correct way 
to look at the monodromy is to view it as a homomorphism from $\pi_1(Q) \cong \bbZ$ to $\bbR^n/Z_\rho$.
\begin{figure}[htb]
\begin{center}
\includegraphics[width=80mm]{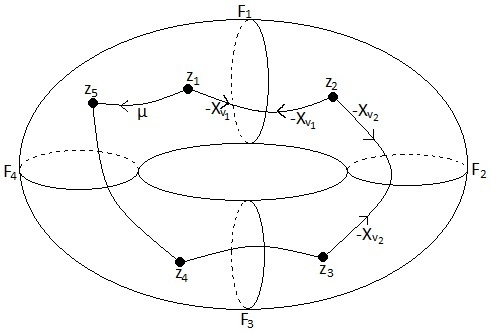}
\caption{Monodromy $\mu$ when $Q \cong \bbS^1$.}
\label{fig:monodromyS1}
\end{center}
\end{figure}

\begin{thm} \label{thm:class-marked}
1) If $(Q_{\text{marked}}, Z)$ is a pair of marked graph and lattice which satisfies the conditions ($C_i$)-($C_{iv}$) above,
then they can be realized as the marked graph and the isotropy group of a nondegenerate action of $\bbR^n$
of toric degree $n-1$ on a compact $n$-manifold. Moreover, if $Q$ is a circle then any monodromy 
element $\mu \in \bbR^n/Z$ can also be realized. 

2) a) In the case when $Q$ is an interval, then any two such actions having the same 
$(Q_{\text{marked}}, Z)$-invariant are isomorphic. b) In the case when $Q$ is a circle, 
then any two actions having the same ($Q_{\text{marked}}, Z, \mu$) are isomorphic.
\end{thm}

\begin{proof}
1) The proof is by surgery, i.e. gluing of linearized pieces given by Theorem \ref{thm:semi-localform}. There is no 
obstruction to doing so.

2a) If there are 2 different actions $(M_1, \rho_1)$ and $(M_2, \rho_2)$ with the same marked graph $(S_{\text{marked}}, Z)$, 
then one can construct an isomorphism $\Phi$ from $(M_1, \rho_1)$ to $(M_2, \rho_2)$ as follows:
Take $z_1 \in M_1$    and $z_2 \in M_2$ such that $z_1$ and $z_2$ project to the 
same regular point on $S_{\text{marked}}$. Put $\Phi (z_1) =z_2$. Extend $\Phi$ to $\cO_{z_1}$ by the formula 
$\Phi (\rho_1(\theta,z_1)) = \rho_2(\theta, z_2).$
Then extend $\Phi$ to rest of $M_1$ by the reflection principle and the continuity principle.

2b) The proof is similar to that of assertion 2a).

\end{proof}

\subsubsection{Three-dimensional case of toric degree $1= 3-2$}

Consider an action $\rho: \bbR^3 \times M^3 \to M^3$ of toric degree 1.
Let $z \in \cO_z$ be a point in a singular orbit of $\rho$. Denote the
HERT-invariant of $z$ by $(h,e,r,t)$, and by $k=\rank_{\bbZ_2}G_z$ the rank over $\bbZ_2$
of the twisting group $G_z$ of $\rho$ at $z$. According to the results of the previous 
subsections, we have following constraints on the nonnegative integers $h,e,r,t,k$:
\begin{equation}
h + 2e + r + t = 3,\ \ e+t = 1,\ \ e + h \geq 1,\ \ k \leq \min (h,t). 
\end{equation}
In particular, we must have $k\leq 1$, i.e. the twisting group $G_z$
is either trivial or isomorphic to $\bbZ_2$.

Taking the above constraints into account, we have the following 
full list of possibilities for the singular point $z$, together with their abbreviated names:

I. $(h) \quad h = 1, e=0, r=1, t=1, G_z = \{0\}$

II. $(h_t) \quad h = 1, e=0, r=1, t=1, G_z = \bbZ_2$

III. $(e) \quad h = 0, e=1, r=1, t=0, G_z = \{0\}$

IV. $(h-h) \quad h = 2, e=0, r=0, t=1, G_z = \{0\}$

V. $(h-h_t) \quad h = 2, e=0, r=0, t=1, G_z = \bbZ_2$
acting by the involution $(x_1,x_2)\mapsto (-x_1,x_2)$

VI. $(h-h)_t \quad h = 2, e=0, r=0, t=1, G_z = \bbZ_2$
acting by the involution $(x_1,x_2)\mapsto (-x_1,-x_2)$

VII. $(e-h) \quad h = 1, e=1, r=0, t=0, G_z = \{0\}$

In the above list, $(h)$ means hyperbolic non-twisted, $(h-h)_t$ means a joint twisting
of a product of 2 hyperbolic components, and so on.

The local structure of the corresponding 2-dimensional quotient space $Q^2=M^3/\rho_{\bbT}$
(together with the traces of singular orbits on $M^3$) is described in Figure \ref{fig:7types}.
\begin{figure}[htb] 
\begin{center}
\includegraphics[width=120mm]{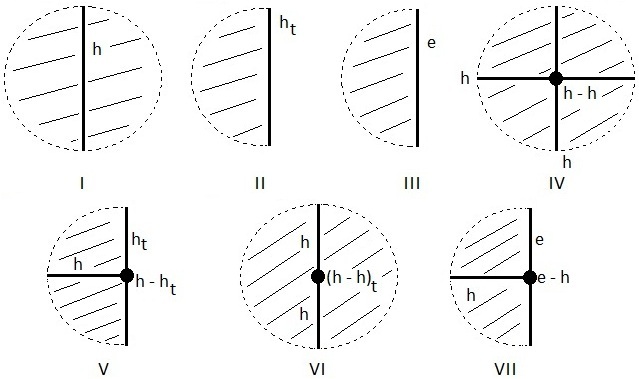}
\caption{The 7 types of singularities of $\bbR^3$-actions of toric degree 1 on 3-manifolds.}
\label{fig:7types}
\end{center}
\end{figure}

Remark that, in Case VI, locally $Q\cong D^2/\bbZ_2$ is homeomorphic but not diffeomorphic to a disk.
In the other cases, $Q$ can be viewed locally as either a disk (without boundary) or a haft-disk (with boundary)
but it cannot be a corner.

Globally, the quotient space $Q$ can be obtained by gluing copies of the above 
7 kinds of local pieces together, in a way which respects the letters (e.g. an edge marked $e$
will be glued to an edge marked $e$, an edge marked $h_t$ will be glued to an edge marked $h_t$).

Notice, for example, that Case II and Case III in the above list are different but have diffeomorphic quotient spaces. 
To distinguish such situations, we must attach letters to the singularities, which describe the corresponding types
of singularities coming from $(M^3,\rho)$. The quotient space $Q$ together with these letters on its graph of singular
orbits will be called the {\bf typed quotient space} and denoted by $Q_{typed}$.

\begin{thm} \label{thm:n-2a}
 1) Let  $(Q_{typed},\rho_\bbR)$ be the typed quotient space of $(M^3,\rho)$, where $\rho$ is of toric degree 1
and $M^3$ is a 3-manifold without boundary. Then each singularity of $Q_{typed}$ belongs to one of the seven types I--VII
listed above.

2) Conversely, let $(Q_{typed},\rho_\bbR)$ be a 2-orbifold together with a totally hyperbolic action $\rho_\bbR$ on it, and
together with the letters on the graph of singular orbits, such that the singularities of $Q_{typed}$ belong to the above
list of seven types I--VII. Then there exists $(M^3,\rho)$ of toric degree 1 which admits  $(Q_{typed},\rho_\bbR)$
as its quotient. Moreover, the $\bbT^1$-equivariant diffeomorphism type of $M^3$ is completely determined by $Q_{typed}$.
\end{thm}

\begin{proof}
 1) It was shown above that the list I--VII is complete in the case of dimension 3, due to dimensional constraints.

2) When the toric degree is 1, assuming that $Z_\rho \cong \bbZ$ is fixed in $\bbR^3$, 
because $Z_\rho$  has only 1 dimension and doesn't allow multiple choices, 
there is a unique choice of isotropy groups in this case. 
The second part of the theorem now follows from Theorem \ref{cor:EquivariantUniqueness}.
\end{proof}

\begin{figure}[htb] 
\begin{center}
\includegraphics[width=120mm]{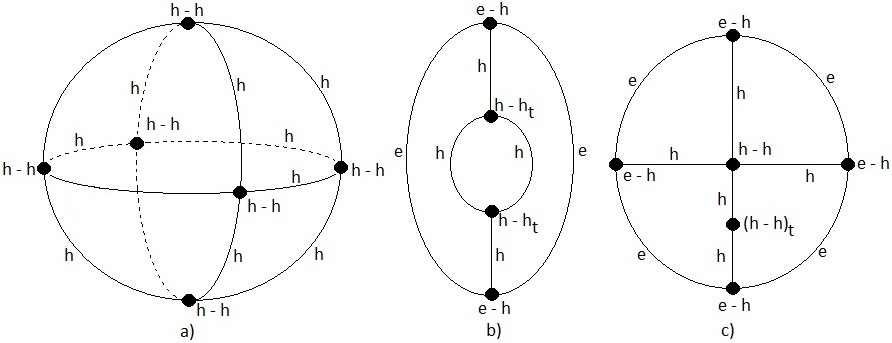}
\caption{Example of $Q^2$ for $n=3, t(\rho) =1$.}
\label{fig:Q2-7kinds}
\end{center}
\end{figure}

Some examples of realizable $Q_{typed}$ which can be obtained by gluing the above 7 kinds of pieces
are shown in Figure \ref{fig:Q2-7kinds}. Notice that $Q_{typed}$ may be without boundary 
(as in Figure \ref{fig:Q2-7kinds}a) or with boundary 
(Figure \ref{fig:Q2-7kinds}b and \ref{fig:Q2-7kinds}c). 
The boundary components of $Q_{typed}$ correspond to the orbits of type $e$ (elbolic)
and $h_t$ (hyperbolic twisted). In the interior of $Q_{typed}$, one may have edges of type $h$ (hyperbolic non-twisted)
and singular points of type $h-h$ or $(h-h)_t$. In Figure \ref{fig:Q2-7kinds}c,
$Q_{typed}$ is not a smooth manifold (though it is homeomorphic to a disk). The branched 2-covering of 
Figure \ref{fig:Q2-7kinds}c is shown in Figure \ref{fig:branched-Double} ($\bbZ_2$ acts by rotating $180^\circ$ around 0).
It is easy to see that, the 3-manifolds corresponding to the situations a), b) c) 
in this example are $\bbS^2 \times \bbS^1$, $\bbR \bbP^2 \times \bbS^1$, and $\bbR \bbP^3$ respectively.

\begin{figure}[htb] 
\begin{center}
\includegraphics[width=80mm]{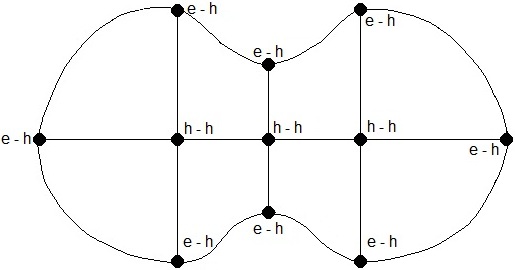}
\caption{Branched double covering of Figure \ref{fig:Q2-7kinds}c.}
\label{fig:branched-Double}
\end{center}
\end{figure}

\begin{remark} If $M^3$ admits an action of $\bbR^3$ of toric degree 1, 
then $M^3$ is a graph-manifold in the sense of Waldhausen, see, e.g. \cite{JacoShalen-GraphManifolds}. 
As was observed by Fomenko \cite{Fomenko-MorseTheory1986}, graph-manifolds are also precisely
those manifolds which can appear as isoenergy 3-manifolds in an integrable Hamiltonian system
with 2 degrees of freedom.
\end{remark}

\subsubsection{The case of dimension $n \geq 4$ with toric degree $n-2$}

When the dimension $n$ is at least 4 and the toric degree is $n-2 \geq 2$, we have the following 3 new types
of singularities, in addition to the 7 types listed above (See Figure \ref{fig:add3types}):

VIII. $(h_t-h_t) \quad h = 2, e=0, r=0, t=n-2, G_z = \bbZ_2 \times \bbZ_2$ acting separately on the two hyperbolic components.

IX. $(e-h_t) \quad h = 1, e=1, r=0, t=n-3, G_z = \bbZ_2$.

X. $(e-e) \quad h = 0, e= 2, r=0, t=n-4, G_z = \{0\}$.

\begin{figure}[htb] 
\begin{center}
\includegraphics[width=100mm]{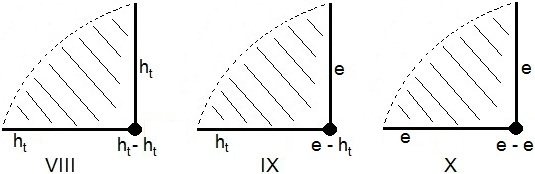}
\caption{The additional 3 possible types of singularities for actions of toric degree $n-2$ when $n\geq 4$.}
\label{fig:add3types}
\end{center}
\end{figure}

\begin{thm} \label{thm:n-2b}
Let  $(Q_{typed},\rho_\bbR)$ be the typed quotient space of $(M^n,\rho)$, where $M^n$ is a compact manifold and 
 $\rho: \bbR^n \times M^n \to M^n$ is nondegenerate of toric degree $n-2$. 
Then each singularity of $Q_{typed}$ belongs to one of the ten types I--X
listed above. Conversely, let $(Q_{typed},\rho_\bbR)$ be a 2-orbifold with boundary and corners together with 
a totally hyperbolic action $\rho_\bbR$ on it, and
together with the letters on the graph of singular orbits, such that the singularities of $Q_{typed}$ belong to the above
ten types I--X. Then for any $n \geq 4$ there exists $(M^n,\rho)$ of toric degree 
$n-2$ which admits  $(Q_{typed},\rho_\bbR)$ as its quotient. 
\end{thm}

\begin{proof}
The main point of the proof is to show that  one can choose compatible isotropy groups, but it is a simple excercise.
Remark that, unlike the case of dimension 3, when $n \geq 4$ the typed quotient $Q_{typed}$ 
does not determine the diffeomorphism type of the 
manifold $M$ completely, because there are now multiple choices for the isotropy groups.
\end{proof}

\begin{figure}[htb] 
\begin{center}
\includegraphics[width=50mm]{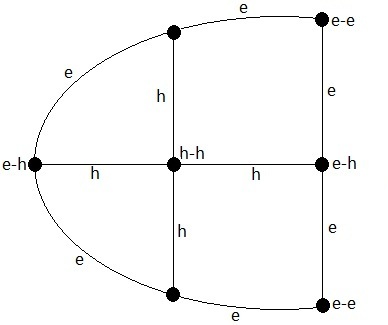}
\caption{Example of $Q = M^4/\bbT^2$.}
\label{fig:exampleM4T2}
\end{center}
\end{figure}

An example of the quotient space $Q$, which can't appear for $n=3$ but can appear for $n \geq 4$, 
is shown in  Figure \ref{fig:exampleM4T2}.

\subsection{The monodromy}
\label{subsection:monodromy}

\subsubsection{Definition of monodromy}

In the classification of actions of toric degree $n-1$, we have encountered a global 
invariant called the monodromy. It turns out that the monodromy can also be defined for any nondegenerate 
action $\rho: \bbR^n \times M^n \to M^n$ of any toric degree, and is one of the main invariants of the action.

Choose an arbitrary regular point $z_0 \in (M^n, \rho)$, and a loop 
$\gamma : [0,1] \to M^n, \gamma(0) = \gamma(1) = z_0$. By a small perturbation 
which does not change the homotopy class of $\gamma$, we may assume that
$\gamma$ intersects the $(n-1)$ singular orbits of $\rho$ transversally (if at all), 
and does not intersect the orbits of dimension $\leq n -2$.
Denote by $p_1, \hdots, p_m$ ($m \geq 0$) the singular points of corank 1 on the loop $\gamma$,
and $\sigma_1, \hdots, \sigma_m$ the associated reflections of the singular hypersurfaces which contain 
$p_1, \hdots, p_m$ respectively as given by Theorem \ref{thm:Reflection}. 
Put $z_1 =\sigma_1(z_0), z_2 =\sigma_2(z_1), \hdots, z_m =\sigma_m(z_{m-1})$. 
(The involution $\sigma_0$ can be extended from a small neighborhood of $p_1$
to $z_0$ in a unique way which preserves $\rho$, and so on). Then $z_m$ lies in 
the same regular orbit as $z_0$, so there is a unique element
$\mu = \mu(\gamma ) \in \bbR^n/Z_\rho$ such that $z_m = \rho(\mu(\gamma), z_0)$.It turns out
that $\mu(\gamma )$ depends only on the homotopy class of $\gamma$, and gives rise to
a group homomorphism  
\begin{equation}
\mu : \pi_1(M^n,z_0) \to \bbR^n/Z_\rho. 
\end{equation}
Moreover, due to the commutativity of $\bbR^n/Z_\rho$, $\mu$ does not depend on the choice of $z_0$ and can be
viewed as a homomorphism from the first homology group 
$H_1(M^n, \bbZ)$ to $\bbR^n/Z_\rho$, which for simplicity will also be denoted by $\mu$:
\begin{equation}
\mu: H_1(M^n, \bbZ) \to \bbR^n/Z_\rho.
\end{equation}

\begin{defn}
The above homomorphisms $\mu : H_1(M^n, \bbZ) \to \bbR^n/Z_\rho$ and $\mu : \pi_1(M^n) \to \bbR^n/Z_\rho$
are called the {\bf monodromy} of the action $\rho: \bbR^n \times M^n \to M^n$.
\end{defn}

\begin{remark} 
The above monodromy is a continuous invariant, and is completely different from the notions
of monodromy defined by Duistermaat \cite{Duistermaat-globalaction-angle1980} and 
Zung \cite{Zung-Integrable2003} for integrable Hamiltonian systems, which are
discrete invariants.
\end{remark}

\subsubsection{Monodromy and twisting groups}

A simple but important observation is that the twisting groups are subgroups of the monodromy group, i.e.
the image of $\pi_1(M^n)$ by $\mu$ in $\bbR^n/Z_\rho$:

\begin{thm}[Twistings and monodromy] \label{thm:TwistingMonodromy}
For any point $z \in M^n$, we have
\begin{equation}
G_z \subseteq Im(\mu),
\end{equation}
where $G_z = (Z_\rho(z) \cap Z_\rho \otimes \bbR)/Z_\rho$ is the twisting group
of the action $\rho$ at $z$, and $Im(\mu) = \mu(\pi_1(M^n)) \subseteq \bbR^n/Z_\rho$
is the image of $\pi_1(M^n)$ by the monodromy map $\mu$. In particular,
if $M^n$ is simply-connected, then $Im(\mu)$ is trivial, and $\rho$ has no twisting.
\end{thm}
\begin{proof}
Let $q \in M^n, (w \mod Z_\rho) \in G_z$, and $z_0$ be a regular point close enough to $z$.
Consider the loop $\gamma: [0,1] \to M^n$  defined as follows:
\begin{itemize}
\item $\gamma(t) = \rho(2tw,z_0) \quad \forall 0 \leq  t \leq \frac{1}{2}$
\item $\gamma(t)$ for $\frac{1}{2} \leq  t \leq 1$ is a path from $\rho(w,z_0)$ to $z_0$
in a small neighborhood of $z$.
\end{itemize}
The one verifies, using the definition of the monodromy and the semi-local normal form theorem,
that $\mu([\gamma]) = w \mod Z_\rho$.
\end{proof}
The proof of the above theorem shows that the monodromy map $\mu : H_1(M^n,\bbZ) \to \bbR^n/Z_\rho$
satisfies the following compatibility condition (*) with the isotropy groups:

\vspace{0.3cm}
(*) \emph{ If $[\gamma] \in H_1(M^n,\bbZ)$  can be represented by a loop of the type 
$\{\rho(tw,p) | t \in [0,1]\}$ where $p \in M^n, w \in Z_\rho(p) \cap Z_\rho\otimes \bbR$, then
$\mu([\gamma])= w \mod Z_\rho$.}
\vspace{0.3cm}

In particular, If $[\gamma] \in H_1(M^n,\bbZ)$  can be represented by a loop of the type 
$\{\rho(tw,p) | t \in [0,1]\}$ where $w \in Z_\rho$, then $\mu([\gamma])=0$.

\begin{remark}
We don't know yet if $G_{\text{torsion}}$ is completely generated by the twisting elements or not in general. 
\end{remark}

\subsubsection{Changing of monodromy} The torus action $\rho_\bbT$ induces a natural
homomorphism 
\begin{equation}
\tau: Z_\rho \to H_1(M^n,\bbZ),
\end{equation}
which associates to each element $w \in Z_\rho$ the homology class of a loop of the type 
$\{\rho(tw,z) | t \in [0,1]\}$ in $H_1(M^n,\bbZ)$, which does not depend on the choice of $z$
in $M^n$. The composition of $\tau$ with $\mu$ is trivial, because the image of the homology class of
any such loop under the monodromy map is zero. Thus we can also view the monodromy as a homomorphism
from $H_1(M^n,\bbZ)/\tau(Z_\rho)$ to $\bbR^n/Z_\rho$, which, by abuse of language, we will also denote by $\mu$: 
\begin{equation}
\mu : H_1(M^n,\bbZ)/\tau(Z_\rho) \to \bbR^n/Z_\rho.
\end{equation}

According to the structural theorem for finitely generated Abelian groups, we can write
\begin{equation}
H_1(M^n,\bbZ)/\tau(Z_\rho) = G_{\text{torsion}}\oplus G_{\text{free}},
\end{equation}
where $G_{\text{torsion}} \subseteq H_1(M^n,\bbZ)/\tau(Z_\rho)$ is its torsion part, and
$G_{\text{free}} \cong \bbZ^k$, where $k = \rank_\bbZ \big(H_1(M^n,\bbZ)/\tau(Z_\rho)\big)$,
is a free part complementary to $G_{\text{torsion}}$.
This decomposition of $H_1(M^n,\bbZ)/Im(Z_\rho)$ gives us a decomposition of $\mu$:
\begin{equation}
\mu = \mu_{\text{torsion}}\oplus \mu_{\text{free}},
\end{equation}
where $\mu_{\text{torsion}} : G_{\text{torsion}} \to \bbR^n/Z_\rho$
is the restriction of $\mu$ to the torsion part $G_{\text{torsion}}$, and
$\mu_{\text{free}}$ is the restriction of $\mu$ to $G_{\text{free}}$.

Notice that $\mu_{\text{torsion}}$ is not arbitrary, but must satisfy the above 
compatibility condition (*) with the twisting groups. On the other hand, $\mu_{\text{free}}$
can be arbitrary. More precisely, we have:

\begin{thm} \label{thm:trans-monodromy}
With the above notations, assume that $\mu_{\text{free}}' : G_{\text{free}} \to \bbR^n/Z_\rho$
is another arbitrary homomorphism from $G_{\text{free}}$ to $\bbR^n/Z_\rho$. Put 
\begin{equation}
 \mu' = \mu_{\text{torsion}}\oplus \mu_{\text{free}}':  H_1(M^n,\bbZ)/\tau(Z_\rho) \to \bbR^n/Z_\rho.
\end{equation}
Then there exists another nondegenerate action $\rho': \bbR^n \times M^n \to M^n$, 
which has the same orbits as $\rho$ and the same isotropy group at each point of $M^n$ 
as $\rho$, but whose monodromy is $\mu'$.
\end{thm}

See \cite{ZungMinh_Rn2014} for the proof, which is based on a covering of $M$ which trivializes
the monodromy.

\begin{remark}
In Theorem \ref{thm:trans-monodromy}, it is possible to change $\mu_{\text{torsion}}$ also to another
homomorphism $\mu_{\text{torsion}}' : G_{\text{torsion}} \to \bbR^n/Z_\rho$. Then the construction of 
the proof still works, but the new action $\rho'$ will not have the same isotropy groups as $\rho$ at 
twisted singular orbits in general, and even the diffeomorphism type of $M'$ may be different from $M$,
because the new action of $\pi_1(M^n,z_0)/Im(Z_\rho)$ will not be isotopic to the old one.
\end{remark}

\subsubsection{Monodromy under reduction}

Even though $Q = M^n / \rho_\bbT$ is just an orbifold in general, we can still define the monodromy map
$\mu_{\rho_{\bbR}} : H_1(Q,\bbZ) \to \bbR^n/ (Z_\rho \otimes \bbR) \cong \bbR^{n-t(\rho)}$ of the action $\rho_\bbR$
on $Q$, just like the case of actions on manifolds. The following proposition is an immediate consequence of the
definition of monodromy:
\begin{prop}
 We have the following natural commutative diagramme:
\begin{equation}\label{eq:diag-monodromy}
\xymatrix{ 
H_1(M^n,\bbZ) \ar^{\mu_{\rho}}[r]\ar^{proj.}[d]& \bbR^n/Z_\rho \ar^{proj.}[d]\\
H_1(Q,\bbZ) \ar^{\mu_{\rho_{\bbR}}}[r]        &    \bbR^n/ (Z_\rho \otimes \bbR)\\
} 
\end{equation}
where $proj.$ denotes the natural projection maps.
\end{prop}

\subsection{Totally hyperbolic actions}
\label{subsection:hyperbolic}

\subsubsection{Hyperbolic domains and their fans}

Recall that a nondegenerate action $\rho: \bbR^n \times M^n \to M^n$
is {\bf totally hyperbolic} if its toric degree is zero, or equivalently, its regular orbits 
are diffeomorphic to $\bbR^n$. Each such orbit will be called a 
{\bf hyperbolic domain} of the action on the manifold. 

Let $\cO$ be a hyperbolic domain of a totally hyperbolic action
 $\rho: \bbR^n \times M^n \to M^n$. According to the toric degree formula (Theorem \ref{thm:HERT-toricdegree}),
when the toric degree of the action is 0, then every orbit is diffeomorphic to some $\bbR^k$ ($0 \leq k \leq n$),
and so we have a cell decomposition of $M^n$, whose cells are the orbits of $\rho.$ Denote by 
$\bar{\cO}$ the closure of $\cO$ in $M^n$ and call it a {\bf closed hyperbolic domain}. 
Then $\bar{\cO}$ also admits a cell decomposition by the
orbits of $\rho$.

Fix an arbitrary point $z_0 \in \cO$. For each orbit $\cH$ of $\rho$ in $\bar{\cO},$ we will denote
by
\begin{equation}
C_{\cH} = \{ w \in \bbR^n \ |\ \lim_{t \to +\infty}\rho(-tw,z_0) \in \cH \} 
\end{equation}
the set of all elements $w \in \bbR^n$ such that the flow of the action $\rho$ 
through $z_0$ in the direction $-w$ tends to a point in $\cH$. It is clear that if $\cH$
and $\cK$ are two different orbits in $\bar{\cO}$ then $C_\cH$ and $C_\cK$ are disjoints.
Using local normal forms, one can prove the following:

\begin{prop} \label{prop:simpli_cone}  With the above notations, we have: \\
1) $C_\cH$ does not depend on the choice of $z_0\in \cO$. \\
2) $C_\cO = \{0\}$ and $C_{F_i} = \bbR_{>0}.v_i$ for each $(n-1)$-dimensional orbit $F_i \subset \bar \cO$, 
where $v_i \in \bbR^n$ is the vector associated to $F_i$ with respect to the action $\rho$.\\
3) If $w \in C_\cH$ then $X_w = 0$ on $\cH$, where  $X_w = \frac{d}{dt} \rho(tw,.)|_{t=0}$ is the
generator of $\rho$ associated to $w.$ \\
4) $\bar C_\cH$ is a simplicial cone in $\bbR^n$ (i.e. a convex cone whose base is a simplex, that is a $k$-dimensional
polytope with exactly $k+1$ vertices for some $k$) 
and $\dim C_\cH + \dim \cH = n$. \\
5) $C_\cK \subset  \bar C_\cH$ if and only if $\cH \subset \bar \cK$
and in that case $C_\cK$ is a face of $\bar C_\cH$. \\
6) If $\bar \cO$ is compact, then the family ($C_\cH; \cH \subset \bar \cO$)
is a partition of $\bbR^n$. 
\end{prop}

In the literature, a partition of $\bbR^n$ into a finite disjoint union of simplicial cones starting at the origin 
is called a \emph{complete fan}, see, e.g. \cite{Ewald-Combinatorial1996}, \cite{Ishida-toric}. More precisely, we have
the following definition:

\begin{defn} \label{def:fan}
A {\bf fan} in $\bbR^n$ is a set of data $(C_\cH, v_i)$, where $\cH$ and $i$ are indices, such that:

i) The family $(C_\cH)$ is a finite family of disjoint subsets of $\bbR^n$

ii) The closure $\bar  C_\cK$ of each  $C_\cK$ is a simplicial cone in $\bbR^n$ 
whose vertex is the origin of $\bbR^n,$
and $\bar  C_\cK \backslash C_\cK$ is the boundary of the cone $\bar  C_\cK$.

iii) If $\bar  C_\cK \backslash C_\cK \not= \emptyset$ (i.e. $C_\cK \not= \{0\}$) then each face of 
$\bar C_\cK$ is again an element of the family $(C_\cH)$.

iv) The $(v_i)$ are vectors in $\bbR^n$ which lie on 1-dimensional cones in the family $(C_\cH)$,
such that each 1-dimensional $C_{\cK_i}$ contains exactly one element $v_i$: $C_{\cK_i} = \bbR_{>0}.v_i$.

v) If moreover $(C_\cH)$ is a finite partition of $\bbR^n$, i.e. 
$\bbR^n$ is the disjoint union of this family  $(C_\cH)$, then we say that $(C_\cH, v_i)$
is a {\bf complete fan}.
\end{defn}

Proposition \ref{prop:simpli_cone} tells us exactly that to each compact closed hyperbolic domain $\bar \cO$ there is
a naturally associated complete fan of $\bbR^n$, which is an invariant of the action. If $\bar\cO$ is not compact, 
then the cones $(C_\cH)$ do not fill the whole $\bbR^n$, and one has an incomplete fan in that case.
Figure \ref{fig:thefan} is an illustration of the construction of the associated fan for a hyperbolic domain.
\begin{figure}[htb]
\begin{center}
\includegraphics[width=90mm]{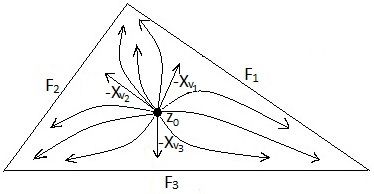}
\caption{The fan at $T_{z_0}M^n \cong \bbR^n$.}
\label{fig:thefan}
\end{center}
\end{figure}

The following theorem shows that, 
conversely, any complete fan can be realized, and is the full invariant of the action on a 
compact closed hyperbolic domain.

\begin{thm} 
\label{thm:classificationByfan}
 1) Let $(C_\cH, v_i)$ be a  fan of $\bbR^n$. Then there exists a totally hyperbolic action  
$\rho: \bbR^n \times M^n \to M^n$ on a  manifold $M^n$ 
with a hyperbolic domain $\cO$ such that the associated fan to $(\bar \cO, \rho)$ is $(C_\cH, v_i)$.
If the fan is complete then $M^n$ can be chosen to be compact without boundary.

2) If there are two closed hyperbolic domains $(\bar \cO_1, \rho_1)$ and $(\bar \cO_2, \rho_2)$ of two
actions $\rho_1$ and $\rho_2$, which have the same associated complete fan $(C_\cH, v_i)$, then there is 
a diffeomorphism from $\bar \cO_1$ to $\bar \cO_2$ which intertwines $\rho_1$ and $\rho_2$.
\end{thm}

\begin{proof}
1) One uses the gluing method to construct an abstract hyperbolic 
domain associated to the fan, and then use reflections to glue pieces
isomorphic to that domain together into a closed manifold. More precisely, take an arbitrary closed 
hyperbolic domain $\bar\cO$ and take $2^k$ identical copies of it, indexed by  numbers 
$a \in \{1, \hdots, 2^k\}$, where $k$ is the number of $(n-1)$-dimensional facets of $\bar\cO$. Glue these identical
copies together to get a global manifold together with a totally hyperbolic action on it by the following
rule: $\bar\cO_a$ will be glued to $\bar\cO_b$ along the $i$-th facet if and only if $|a-b| = 2^{i-1}.$

2) Take any two points $z_1 \in \cO_1$ and $z_2 \in \cO_2$. Define 
\begin{equation}
\Phi (z_1) =z_2, \Phi(\rho_1(\theta,z_1)) = \rho_2(\theta,z_2)
\end{equation}
for all $\theta \in \bbR^n$, and then extend $\Phi$ to the boundary of $\bar \cO$ 
by continuity. The fact that $(\bar \cO_1, \rho_1)$ and $(\bar \cO_2, \rho_2)$ have the same associated 
complete fan ensures that the constructed map $\Phi : \bar \cO_1 \to \bar \cO_2$ is a diffeomorphism, which sends 
$\rho_1$ to $\rho_2$. 
 \end{proof}

\subsubsection{The topology of closed hyperbolic domains}

Since closed hyperbolic domains are classified by fans, their topology (together with the cell decomposition
by the orbits) is also completely determined by the corresponding fans, which are combinatorial objects.
By starting from fans, one can show that any closed hyperbolic domain $\bar{\cO}$ is {\bf contractible}. It is easy to 
see that any convex simple polytope in $\bbR^n$ (i.e. a convex polytope which has exactly $n$ edges at each vertex)
can be realized by a complete fan (whose vectors are orthogonal to the facets of the polytope and pointing 
towards the facets from inside). A natural question arises: is it true that any compact close hyperbolic
domain $\bar{\cO}$ is also difeomorphic to a convex simple polytope? The following theorem, which was 
pointed out by Ishida, Fukukawa and Masuda in \cite{Ishida-toric}
in the context of topological toric manifolds (see Subsection \ref{subsection:elbolic}) 
gives an answer to this question:

\begin{thm} \label{thm:diff-to-polytope}
Any compact closed hyperbolic domain $\bar \cO$ of dimension $n \leq 3$ is diffeomorphic to a convex
simple polytope. If $n \geq 4$ then there exists a compact closed hyperbolic domain $\bar \cO$ of dimension $n$
which is not diffeomorphic to a polytope. 
\end{thm}

The case $n=2$ of the above theorem is obvious. The case $n=3$ is a consequence of the classical Steinitz theorem. 
When $n=4$ of higher, there are counterexamples:
The first known counterexample comes from the so-called Barnette's sphere \cite{Barnette-Diagram1970}.
The Barnette's sphere is a simplicial complex whose ambient space is a 3-dimensional sphere $S^3$,
but which cannot be realized as the boundary of a convex simplicial polyhedron in $\bbR^3$ for some reasons of 
combinatorial nature. It is known \cite{Ewald-Combinatorial1996} that Barnette's sphere can be realized 
as the base of a complete fan in $\bbR^4$,
which we will call the Barnette fan. Take the closed hyperbolic domain $\bar\cO$ given by this Barnette fan. 
Then $\bar \cO$ cannot be
diffeomorphic to a convex simple 4-dimensional polytope, because if there is such a polytope, 
then the boundary of the simplicial polytope dual to it will be a realization of the Barnette's sphere, 
which is a contradiction.

\subsubsection{Totally hyperbolic actions in dimension 2} 
\label{subsubsection:hyperbolic_dim2}

The existence of  totally hyperbolic actions on any closed
2-manifold was known to Camacho \cite{Camacho-MorseSmaleAction1973}, who called them ``Morse-Smale 
$\bbR^2$-flows on a 2-manifold''.The following theorem is a slight improvement of Camacho's result:

\begin{thm}[\cite{Camacho-MorseSmaleAction1973,ZungMinh_Rn2014}] 
\label{thm:hyperbolic_dim2}
1) On the sphere $\bbS^2$ there exists a totally hyperbolic $\bbR^2$ action, which has
exactly 8 hyperbolic domains. The number 8 is also the minimal number possible:  
any totally hyperbolic action of $\bbR^2$ on $\bbS^2$ must have at least 8 hyperbolic domains.

2) For any $g \geq  1$, on a closed orientable surface of genus g there 
exists a totally hyperbolic action of $\bbR^2$ which has exactly 4 hyperbolic domains. 
The number 4 is also minimal possible.

3) Any non-orientable closed surface also admits a totally hyperbolic action with 4 
hyperbolic domains, and the number 4 is also minimal possible.
\end{thm}
\begin{proof} As for the existence of totally hyperbolic actions on orientable surfaces, 
one can cut a sphere into 8 trigones, or a surface $\Sigma_g$ of genus $g \geq 1$
into 4 pieces, as shown in Figure \ref{fig:S2Sigma2}, turn one of the pieces into a hyperbolic domain and then extend the 
$\bbR^2$-action to the whole surface by reflections. In the case of a non-orientable surface,
one can embed $\Sigma_g$  (where $g\geq 0$) into $\bbR^3$ in such a way that
it is symmetric with respect to the 3 planes $\{x=0\}$, $\{y=0\}$, $\{z=0\}$, and is cut into $8$ polygones
by these planes (each polygone has $g+3$ edges). Like in the case of $\bbS^2$, we turn one of these $(g+3)$-gones into a
hyperbolic domain and then extend it to $\Sigma_g$ by reflections, in such a way that the $\bbR^2$ action is invariant
with respect to the antipodal involution. Hence this action can be projected to the quotient space of $\Sigma_g$ with respect to this 
involution to become a totally hyperbolic action with exactly 4 hyperbolic domains on a non-orientable surface.

\begin{figure}[htb]
\begin{center}
\includegraphics[width=100mm]{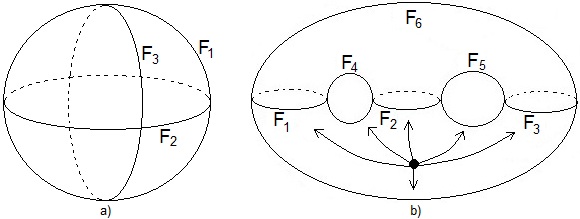}
\caption{Cutting $\bbS^2$ into 8 trigones and cutting $\Sigma_2$ into 4 domains.}
\label{fig:S2Sigma2}
\end{center}
\end{figure}

The minimality of the above numbers is also easy to check, see \cite{ZungMinh_Rn2014}.
\end{proof}

There is a very interesting question of combinatorial nature: what are 
the necessary and sufficient conditions for a graph on a surface $\Sigma$
to be the singular set of a totally hyperbolic action of $\bbR^2$ on $\Sigma$?
We don't know much about this question, but we know that the graph must consist of simple closed curves
which interesect transversally, and besides that there are some other obstructions. For example, 
there does not exist any totally hyperbolic action which contains 3 domains $\cO_1, \cO_2,\cO_3$ 
as in Figure \ref{fig:3domain}a. Indeed, asssume that there is such an action.
Denote by $v_1,v_2,v_3,v_4$ the vectors associated to the curves $F_1,F_2,F_3,F_4$ respectively.
Since $v_1,v_2,v_3$ form the fan of $\cO_1$, we must have: $v_3 = \alpha v_1 + \beta v_2$ for some $\alpha, \beta  < 0$.
Similarly, looking at the fan of $\cO_3$, we have: $v_4 = \gamma  v_1 + \delta  v_2$ for some $\gamma, \delta < 0$.
But looking at the fan of $\cO_2$, we have that either $v_3$ or $v_4$ is a positive linear combination of $v_1$ and $v_2$.
This is contradiction. Similarly, the configuration in Figure \ref{fig:3domain}b is also impossible.

\begin{figure}[htb] 
\begin{center}
\includegraphics[width=120mm]{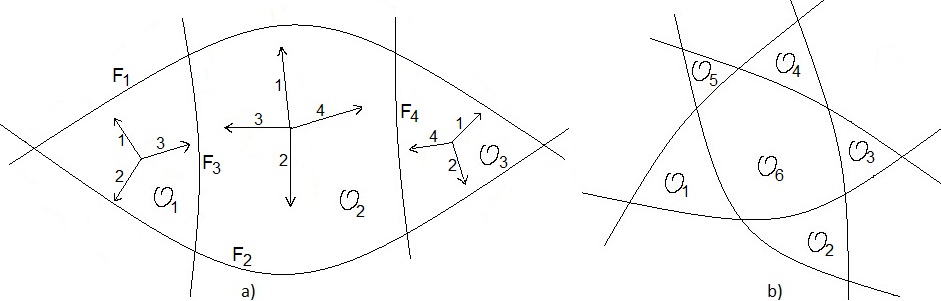}
\caption{Impossible configurations.}
\label{fig:3domain}
\end{center}
\end{figure}

\subsubsection{Classification of totally hyperbolic actions}

As of this writing, we don't know if there is any obstruction to the existence of a totally hyperbolic $\bbR^n$-action on an 
$n$-manifold when $n \geq 3$. We don't even know whether simple manifolds such as 3-dimensional lense spaces
admit a totally hyperbolic $\bbR^n$-action. Nevertheless,
we have the following abstract classification theorem, whose proof is straightforward:

\begin{thm}
\label{class-hyp-invariants}
Nondegenerate totally hyperbolic actions of $\bbR^n$ on a connected $n$-manifold $M^n$
(possibly with boundary and corners) are completely determined by their invariants I1, I2, I3 listed below:

I1) The singular set: smooth invariant hypersurfaces of $M^n$ which intersect transversally and which cut $M^n$
into a finite number of ``curved polytopes'', which are hyperbolic domains of the action.

I2) The family of fans: a fan for each domain, with a correspondence between the cones of the fan and the faces of the closure 
of the domain.  

I3) The monodromy.

In other words, assume that $(M_1^n, \rho_1)$ and $(M_2^n, \rho_2)$ are totally hyperbolic actions,
such that there is a homeomorphism $\varphi : M_1^n \to M_2^n$ which sends hyperbolic domains of $(M_1^n, \rho_1)$
to hyperbolic domains of $(M_2^n, \rho_2)$, such that the monodromy and the associated fans 
are preserved by $\varphi$, then there is a diffeomorphism $\Phi : M_1^n \to M_2^n$ which sends $\rho_1$ to $\rho_2$.
\end{thm}

Remark that the family of fans in the above theorem must satisfy the following compatibility condition: all the fans of
the domains with share the same boundary hypersurface must also share the same vector associated to that hypersurface. This 
additional condition is also sufficient for the set of invariants to be realizable.

\subsection{Elbolic actions and toric manifolds}
\label{subsection:elbolic}

\begin{defn} \label{defn:elbolic}
A nondegenerate action $\rho: \bbR^n \times M^n \to M^n$ is called { \bf elbolic}, if it does not admit any
hyperbolic singularity, i.e. all singular points have only elbolic components.
\end{defn}

In the elbolic case, the singular set of the action is of codimension 2, hence the regular part is connected, 
which implies that the action has just one regular orbit. 
The rest of the proof of the following proposition is also straightforward:

\begin{thm}
Let $\rho: \bbR^n \times M^n \to M^n$ be an elbolic action. Then we have: 

1) $\rho$  has exactly one regular open dense orbit in $M^n$.

2) If the action admits a compact orbit of dimension $k$ then the toric degree is
$t(\rho) = (n+k)/2 \geq n/2$. In particular, If $(M^n,\rho)$ admits a fixed point, then the dimension $n$ is even and $t(\rho) = n/2.$

3) The monodromy of $\rho$ is trivial, and the quotient space $Q =M^n/\rho_\bbT$ of $M^n$ by the induced
torus action $\rho_\bbT$ is a contractible manifold with boundary and corners, on which the reduced action $\rho_\bbR$ is 
nondegenerate totally hyperbolic and has only one regular orbit.

\end{thm}

The case of elbolic actions with a fixed point 
is of special interest in geometry, because of its connection to  toric manifolds. 
Recall that, a {\bf toric manifold} in the sense of complex geometry
is a complex manifold (which is often equipped with a K\"alerian structure, or equivalently, a compatible symplectic
structure) of complex dimension $m$ together with a holomorphic action 
of the complex torus $(\bbC^*)^m$ which has an open dense orbit.
See, e.g., \cite{Audin-Torus2004,Cox-Toric2011} for an introduction to toric manifolds. 
From our point of view, a complex toric manifold has real dimension
$n = 2m$, and the action of $(\bbC^*)^m \cong  \bbR^m \times \bbT^m$ is an elbolic 
nondegenerate $\bbR^{2m}$-action. 
Complex toric manifolds are classfied by their associated fans. So our classification 
of hyperbolic domains (and of the quotient spaces of elbolic actions) are very similar 
to the classification of complex toric manifolds, except that, unlike the complex case,
the vectors of our fans are not required to lie in an integral lattice.

There have been many generalizations of the notion of toric manifolds to the 
case of real manifolds. In particular,  Davis and Januskiewicz introduced
quasi-toric manifolds in 1991 \cite{Davis-Convex1991}. A {\bf quasi-toric
manifold} is a $2m$-dimensional manifold with a almost-everywhere-free action of 
$\bbT^m$ such that the quotient space $M^{2m}/\bbT^m $ is homeomorphic to a simple convex polytope 
(like in the case of complex toric manifolds), and such that
near each fixed point the action is locally isomorphic  to a standard action of $\bbT^m$ on $\bbC^m$. 
Hattori and Masuda introduced torus manifolds in 2003 \cite{Hattori-fan2003}. A {\bf torus manifold} is simply a 
closed connected orientable smooth manifold $M$ of dimension $2m$ with an effective smooth action of $\bbT^m$ having 
a fixed point. Ishida, Fukukawa and Masuda recently introduced the notion of topological toric manifolds in 2010 \cite{Ishida-toric}. A
{\bf topological toric manifold} is a closed smooth manifold $M$ of 
dimension $n = 2m$ with an almost-everywhere-free smooth action of 
$(\bbC^*)^m \cong \bbT^m \times \bbR^m$
which is covered by finitely many invariant open subsets each equivariantly  diffeomorphic to a 
direct sum of complex 1-dimensional linear representation of $\bbT^m \times \bbR^m$. According to the results
of  \cite{Ishida-toric}, topological toric manifolds are \emph{the} right generalization of the notion of toric
manifolds to the category of real manifolds; they have very nice homological 
properties similarly to toric manifolds (see Section 8 of  \cite{Ishida-toric}), and they are also classified by a generalized notion of fans.

It turns out that Ishida--Fukukawa--Masuda's notion of topological toric manifolds is equivalent to
our notion of manifolds admitting an elbolic action whose toric degree is half the dimension of the
manifold. The proof of the following proposition is a simple verification that their conditions and 
our conditions are the same:

\begin{prop}
 A closed manifold $M^{2m}$, together with a smooth  action of $(\bbC^*)^{m} \cong \bbR^m \times \bbT^m$,
is a topological toric manifold if and only if the action (when viewed as an action of $\bbR^{2m}$) 
is elbolic of toric degree $m$.
\end{prop}

In \cite{Ishida-toric}, topological toric manifolds are are classified by the so-called 
{\bf complete non-singular topological fans}, which encode the following data: the complete fan in $\bbR^n$
associated to the reduced totally hyperbolic action $\rho_\bbR$ on the quotient space $Q = M^{2m}/\rho_\bbT$,
and the vector couples associated to corank-2 transversally elbolic orbits (see Definition \ref{def:AssociatedVector}).
These vector couples tell us how to build back $(M^{2m},\rho)$ from $(Q, \rho_\bbR)$.
So one can recover Ishida--Fukukawa--Masuda's classification theorem for topological toric manifolds
from our point of view of general nondegenerate $\bbR^n$-actions on $n$-manfiolds: one can prove this theorem
in the same way as the proof of Theorem \ref{thm:classificationByfan}, by gluing together local pieces equipped with
canonical coordinates and adapted bases. Another very interesting proof, based on the quotient method,
which represents the  topological toric manifold $(M^{2m},\rho)$ as a 
quotient of another global object, is given in \cite{Ishida-toric}. (The quotient method is also discussed in
\cite{Audin-Torus2004,Cox-Toric2011} for the construction of toric manifolds).


\vspace{0.5cm}

\begin{thebibliography}{10}
%\baselineskip0cm
%\parskip-0.1cm
\small{


\bibitem{Audin-Torus2004}
M. Audin, Torus actions on symplectic manifolds. Second revised edition. 
Progress in Mathematics, 93. Birkh\"auser Verlag, Basel, 2004. viii+325 pp. 


\bibitem{AyoulZung-Galois2010}
M. Ayoul, N.T. Zung, {\it Galoisian obstructions to non-Hamiltonian integrability}, Comptes Rendus Mathématiques,
348 (2010), Issue 23, 1323--1326.


%\bibitem{Babelon-Book2003}
%O. Babelon, D. Bernard, M. Talon, Introduction to Classical Integrable Systems, Cambridge monographs on mathematical physics, 2003.

\bibitem{Barnette-Diagram1970}
D. Barnette, {\it Diagrams and Schlegel diagrams}, 1970 Combinatorial 
Structures and their Applications (Proc. Calgary Internat. Conf., Calgary, Alta.) pp. 1–4 Gordon and Breach, New York. 


\bibitem{BaCu-Nonholonomic1999}
L. Bates and R. Cushman, \emph{{What is a completely integrable
  nonholonomic dynamical system ?}}, Reports Math. Phys. \textbf{44} (1999),
  no.~1-2, 29--35.

\bibitem{BK-Equivariant2002}
G.R. Belitskii, A.Y. Kopanskii, {\it Equivariant Sternberg-Chen theorem}, 
Journal of Dynamics and Differential Equations, Volume 14 (2002), Number 2, pp. 349--367.

\bibitem{Bochner-Linearization1945}
S. Bochner, \emph{{Compact groups of differentiable
transformations}},
  Annals of Math. (2) \textbf{46} (1945), 372--381.

\bibitem{Bogoyavlenskij-Integrability1998}
O.I. Bogoyavlenskij, \emph{Extended integrability and bi-hamiltonian
systems},  Comm. Math. Phys. \textbf{196} (1998), no.~1, 19--51.

\bibitem{BBM-Hamiltonization2011}
A. V. Bolsinov,  A. V. Borisov,  I. S. Mamaev,
{\it Hamiltonization of non-holonomic systems in the neighborhood of invariant manifolds}, Regular and Chaotic Dynamics,
Volume 16 (2011), Issue 5, pp 443--464.    

\bibitem{BolsinovFomenko-IntegrableBook}
A.V. Bolsinov, A.T. Fomenko, Integrable Hamiltonian Systems: Geometry, Topology, Classification, Chapman \& Hall/CRC, 2004, xvi+730 pp.

%\bibitem{BFO-Integrable2006}
%A.V. Bolsinov, A.T. Fomenko, A.A. Oshemkov, eds., 
%Topological methods in the theory of integrable systems, Cambridge Scientific Publishers, 2006.

\bibitem{BKK-Monodromy2013}
A.V. Bolsinov, A.A. Kilin, A.O. Kazakov, {\it Topological monodromy in nonholonomic systems},
Nelin. Dinam., Vol. 9 (2013), Vol. 2, 203--227.

\bibitem{Borisov-IntegrableRCD2012}
A.V. Borisov et al., Special issue on integrable nonholonomic systems, Regular and Chaotic Dynamics,
17 (2012), Number 2.

\bibitem{Bruno-Local1989}
A.~D. Bruno, {Local methods in nonlinear differential equations},
Springer Series in Soviet Mathematics, Springer-Verlag, Berlin, 1989.
  
  \bibitem{BrWa-NF1994}
A.D. Bruno and S.~Walcher, \emph{Symmetries and convergence of normalizing
  transformations}, J. Math. Anal. Appl. \textbf{183} (1994), 571--576.

\bibitem{Camacho-MorseSmaleAction1973}
C. Camacho, {\it Morse-Smale $\bbR^2$-actions on two-manifolds. Dynamical systems}, 
(Proc. Sympos., Univ. Bahia, Salvador, 1971), pp. 71-74. Academic Press, New York, 1973.

\bibitem{Camacho-Foliations1985}
C. Camacho, A. Lins Neto, Geometric theory of foliations, Birkh\"auser, 1985, 205 pp. 

\bibitem{Chaperon-geometrie1986}
M. Chaperon, Géométrie différentielle et singularités de systèmes dynamiques. 
Astérisque No. 138-139 (1986), 440 pp.

\bibitem{Chaperon-Focus2013}
M. Chaperon, {\it Normalisation of the smooth focus–focus: a simple proof}, Acta Math. Vietnam., 38 (2013),
No. 1, pp 3--9.

\bibitem{Chaperon-RZ2013}
M. Chaperon,{\it A forgotten theorem on $\bbR^k \times \bbZ^m$-actions}, Regular and Chaotic Dynamics,
18 (2013), No. 6, 742--774.

\bibitem{Chen-Vector1963}
K.T. Chen, {\it Equivalence and decomposition of vector fields about an elementary critical
point}, Amer. J. Math., 85 (1963), 693--722.


\bibitem{Chiba-Renormalization2009}
H. Chiba, {\it Extension and Unification of Singular Perturbation Methods for ODEs Based on the
Renormalization Group Method}, SIAM Journal on Applied Dynamical Systems, 2009.

\bibitem{ColinVey-Morse1979}
Y. Colin de Verdier, J. Vey, \emph{Le lemme de Morse isochore}, Topology, 18 (1979), 283-293.


\bibitem{Colin-Singular2003}
Y. Colin~de Verdière, \emph{{Singular Lagrangian manifolds and
semiclassical
  analysis}}, Duke Math. J. \textbf{116} (2003), no.~2, 263--298.


\bibitem{CoDaMo-Moment1988} %needed
M.~Condevaux, P.~Dazord, and P.~Molino, \emph{{Géométrie du moment}},
Séminaire  Sud-Rhodanien I, Publications du d\'epartment de math., Univ. Claude Bernard
  - Lyon I (1988), 131--160.

\bibitem{Cox-Toric2011}
D. A. Cox, J. Little, H. Schenk. Toric varieties, Graduate Studies in Math. vol. 124, AMS,
2011.

\bibitem{CuDu-Focus2001}
R. Cushman, J.J. Duistermaat, {\it Non-Hamiltonian Monodromy}, J. Diff. Equations, Volume 172 (2001), Issue 1, 42–58.

\bibitem{Davis-Convex1991}
M. W. Davis and T. Januszkiewicz, {\it Convex polytopes, Coxeter orbifolds and torus actions}, 
Duke Math. J. 62 (1991), 417-451.

\bibitem{DGJ-Nonholonomic1998}
V. Dragovic, B. Gajic, B. Jovanovic,
{\it Generalizations of classical integrable nonholonomic rigid body systems},
J. Phys. A: Math. Gen., 31 (1998), 9861.

\bibitem{DufourMolino-AA}
J.P. Dufour, P. Molino, {\it Compactification d'actions de $\bbR^n$ et variables action-angle avec singularités},
Publications du Département de Mathématiques (Lyon), 1988, No. 1B, 161--183.


\bibitem{Duistermaat-globalaction-angle1980}
J.J. Duistermaat, {\it On global action-angle variables}, Comm. Pure Appl. Math., 33 (1980),
687-706. 151-167.

%\bibitem{DufourZung-PoissonBook}   
%J.P. Dufour, N.T. Zung, Poisson structures and their normal forms, Progress in Mathematics, Vol. 242, 2005.

\bibitem{Eliasson-Thesis1984}
L.H. Eliasson, {\it Hamiltonian systems with Poisson commuting integrals},
Doctoral Thesis, University of Stockholm, 1984.

\bibitem{Eliasson-Normal1990}
L.H. Eliasson, {\it Normal forms for Hamiltonian systems with Poisson commuting integrals
elliptic case}, Comment. Math. Helv. 65 (1990), no. 1, 4–35.

\bibitem{Ewald-Combinatorial1996}
G. Ewald, Combinatorial Convexity and Algebraic Geometry, Graduate Texts in Math.
vol. 168, Springer-Verlag, New York, 1996.

\bibitem{FaSa-AlmostSymplectic2007}
F. Fass\`o, N. Sansonetto, {\it Integrable almost-symplectic Hamiltonian systems}, J. Math. Phys. 48 (2007),
092902.


\bibitem{FedorovKozlov-Suslin}
Yu. N. Fedorov, V. V. Kozlov, {\it Various aspects of n-dimensional rigid body dynamics}, 
Trans. Am. Math. Soc. Ser. 2. V. 168. 1995. P. 141–171.

%\bibitem{FJ-Chaplygin2006}
%Yu.N. Fedorov; Bo\v zidar Jovanovi\'c, {\it Quasi-Chaplygin Systems and Nonholonimic Rigid Body Dynamics},
%Letters in Mathematical Physics (May 2006), 76 (2-3), pg. 215-230.


\bibitem{Fintushel-Circle1977}
R. Fintushel, {\it Circle actions on simply connected 4-manifolds}, 
Trans. Amer. Math. Soc. 230 (1977), 147–171.

\bibitem{Fomenko-MorseTheory1986}
A.T. Fomenko, {\it Morse theory for integrable Hamiltonian systems}, Soviet Math. Doklady, 
33 (1986), 502--506.

\bibitem{FoMi-Noncommutative1978}
A.T. Fomenko, A.S. Mischenko, {\it Generalized Liouville's method of integration of Hamiltonian systems},
Functional Analysis and its Applications, 12 (1978), Issue 2, 113--121.

\bibitem{GoSt-1Corank1987}
M.~M.~Golubitsky and I.~I.~Stewart, \emph{{Generic bifurcation of
Hamiltonian systems with symmetry}}, Physica D \textbf{24} (1987), 391--405.

\bibitem{GuSt-Convexity1982}
V.~Guillemin and S.~Sternberg, \emph{{Convexity properties of the moment
  mapping}}, Invent. Math. \textbf{67} (1982), 491--513.


\bibitem{Hattori-fan2003}
A. Hattori and M. Masuda, {\it Theory of multi-fans}, Osaka J. Math. 40 (2003), 1-68.


\bibitem{Ishida-toric}
H. Ishida, Y. Fukukawa, M. Masuda, {\it Topological toric manifolds}, 
preprint arXiv:1012.1786, 2010.


\bibitem{Ito-Birkhoff1989}
H. Ito, \emph{Convergence of {B}irkhoff normal forms for integrable
  systems}, Comment. Math. Helv. \textbf{64} (1989), no.~3, 412--461.

\bibitem{Ito-AA1991}
H. Ito, \emph{{Action-angle coordinates at singularities for analytic
  integrable systems}}, Math. Z. \textbf{206} (1991), no.~3, 363--407.

\bibitem{Ito-Birkhoff1992}
H. Ito, \emph{{Integrability of Hamiltonian systems and Birkhoff normal
forms in the simple resonance case}}, Math. Ann. \textbf{292} (1992), no.~3,
  411--444.

\bibitem{JacoShalen-GraphManifolds}
W.H. Jaco, P.B. Shalen, Seifert fibered spaces in 3-manifolds. 
Mem. Amer. Math. Soc. 21 (1979), no. 220.


%\bibitem{Kai-Focus2014}
%J. Kai, {\it Non-Hamiltonian focus-focus ...}, 2014

\bibitem{Jovanovic-Symmetries2008}
B. Jovanovic, {\it Symmetries and Integrability}, 
Publications de l'Institut Mathématique (Beograd), Vol. 84 (2008), No. 98, 1--36.	

\bibitem{Kalashnikov-1Corank1998}
V.~Kalashnikov, \emph{{Generic integrable Hamiltonian systems on a
  four-dimensional symplectic manifold}}, Izv. Math. \textbf{62} (1998), no.~2,
  261--285.

\bibitem{KaKoNe-Birkhoff1998}
T.~Kappeler, Y.~Kodama, and A.~Némethi, \emph{{On the Birkhoff normal form
of a completely integrable Hamiltonian system near a fixed point with resonance}},
  Ann. Scuola Norm. Sup. Pisa Cl. Sci. \textbf{XXVI} (1998), no.~4, 623--661.
  
\bibitem{KaToZu-NLS2009}
T.~Kappeler, T.~Topalov, and N.T. Zung, \emph{Birkhoff normal forms for focusing NLS}, 
Commun. Math. Phys. 285 (2009), 1087--1107.  

\bibitem{KuPe-VeyInfinite2009}
S. Kuksin, G. Perelman, Vey theorem in infinite dimension and its application to KdV, preprint arxiv:0910.0089 (2009)

\bibitem{LMV-AAPoisson2011}
C. Lauent-Gengoux, E. Miranda, P. Vanhaecke, {\it Action-angle coordinates for integrable Hamiltonian
systems on Poisson manifolds}, IMRN, Vol. 2011 (2011), Issue 8, 1839--1869.

\bibitem{Liouville-Torus1855}
J. Liouville, \emph{{Note sur l'int\'{e}gration des \'{e}quations
  differentielles de la dynamique, pr\'{e}sent\'{e}e au bureau des longitudes
  le 29 juin 1853}}, Journal de Math. Pures et Appl. \textbf{20} (1855),
  137--138.
  

\bibitem{Mather-Determinacy1969}
J. Mather, {\it Stability of  $C^\infty$ mappings, III. Finitely determined map-germs}, Publ.
Math. I. H. E. S. 35 (1969), 127–156.

\bibitem{Mineur-AA1935}
H. Mineur, \emph{{Sur les systemes mecaniques admettant $n$ integrales
  premieres uniformes et l'extension a ces systemes de la methode de
  quantification de Sommerfeld}}, C. R. Acad. Sci., Paris \textbf{200} (1935),
  1571--1573 (French).

\bibitem{Mineur-AA1937}
H. Mineur, \emph{{Sur les systemes mecaniques dans lesquels figurent des
  parametres fonctions du temps. Etude des systemes admettant $n$ integrales
  premieres uniformes en involution. Extension a ces systemes des conditions de
  quantification de Bohr-Sommerfeld.}}, Journal de l'Ecole Polytechnique, Série
  III, 143ème année (1937), 173--191 and 237--270.

\bibitem{Miranda-Thesis2003}
E. Miranda, {\it On symplectic linearization of singular Lagrangian foliations}, Doctoral Thesis, 
Universitat de Barcelona, 2003.
  
\bibitem{MirandaZung-NF2004}
E. Miranda, N.T. Zung, {\it Equivariant normal form for 
nondegenerate singular orbits of integrable Hamiltonian systems}, 
Ann. Sci. École Norm. Sup. (4) 37 (2004), no. 6, 819–839.

\bibitem{MRS-Galois2007}
J. Morales-Ruiz, J-P Ramis, C. Simo, {\it Integrability of Hamiltonian systems and differential Galois
groups of higher order variational equations}, Annales Ec. Norm. Sup., 40 (2007), No. 6, 845–884.

\bibitem{Nekhoroshev-Integrable1972}
N.N. Nekhoroshev, \emph{Action-angle variables and their generalizations},
  Trans. Moscow Math. Soc. \textbf{26} (1972), 180--198.

\bibitem{OR_Torus1}
P. Orlik, F. Raymond, {\it Actions of the torus on 4-manifolds, I}, Trans. Amer. Math. Soc. 152 (1970) 531-559.

\bibitem{OR_Torus2}
P. Orlik, F. Raymond, {\it Actions of the torus on 4-manifolds, II}, Topology 13 (1974), 89-112.

\bibitem{Pao-TorusAction1}
P.S. Pao, {\it The topological structure of 4-manifolds with effective torus actions. I}, 
Trans. Amer. Math. Soc. 227 (1977), 279–317.

\bibitem{PelayoSan_Acta2011}
A. Pelayo, San Vu Ngoc, {\it Constructing integrable systems of semitoric type}, Acta Mathematica, 206 (2011), 93--125.

\bibitem{Raissy-Torus2010}
J. Raissy, {\it Torus action in the normalization problem}, Journal of Geometric Analysis, 20 (2010), 472--524.

\bibitem{Roussarie-Asterisque1975}
R. Roussarie, \emph{Modèles locaux de champs et de formes},
Astérisque, Société Mathématique de France, Paris, 1975.

\bibitem{Russmann-NF1964}
H.~Rüssmann, \emph{{\"Uber das Verhalten analytischer Hamiltonscher
  Differentialgleichungen in der N\"ahe einer Gleichgewichtsl\"osung}}, Math.
  Ann. \textbf{154} (1964), 285--300.  


\bibitem{Sternberg}
S. Sternberg, {\it On the structure of local homeomorphisms of Euclidean 
$n$-space, II}, Amer. J. of Math., 80 (1958), 623-631.
  
  
\bibitem{Stolovitch-Singular2000}
L. Stolovitch, {\it Singular complete integrability}, Publications IHES, 91 (2000), 134-210.

\bibitem{Stolovitch-NF2009}
L. Stolovitch, {\it Progress in normal form theory}, 
Nonlinearity, Volume 22, Issue 7, pp. R77-R99 (2009).

\bibitem{Vey-Separable1978}
J. Vey, \emph{Sur certains systèmes dynamiques séparables}, Amer.
J.  Math. \textbf{100} (1978), no.~3, 591--614.

\bibitem{Vey-Isochore1979}
J. Vey, \emph{{Algèbres commutatives de champs de vecteurs isochores}},
{Bull. Soc. Math. France} \textbf{107} (1979), 423--432.

\bibitem{SanWa-Focus2013}
S. Vu Ngoc, Ch. Wacheux, {\it Smooth normal forms for integrable hamiltonian systems near a focus-focus singularity},
Acta Math. Vietam., 38 (2013), No. 1, 107--122.

\bibitem{Wassermann-Symmetry1988}
G.~Wassermann, \emph{{Classification of singularities with compact Abelian
  symmetry}}, Banach Center Publications \textbf{20} (1988), 475--498.
  
%\bibitem{Weinstein-RealSemisimple1987}
%A. Weinstein, {\it Poisson geometry of the principal series and nonlinearizable structures},
%J. Differential Geometry, 25 (1987), 55--73.

\bibitem{Ziglin-Branching1982}
S.L. Ziglin, {\it Branching of solutions and non-existence of first integrals in Hamiltonian mechanics},
Funcional Anal. Appl. 16 (1982), 181--189. 

\bibitem{Zung-Symplectic1996}
N.T. Zung, {\it Symplectic topology of integrable Hamiltonian systems. 
I. Arnold-Liouville with singularities}, Compositio Math. 101 (1996), no. 2, 179-215.

\bibitem{Zung-Degenerate2000}
N.T. Zung, \emph{A note on degenerate corank-one singularities of integrable
  {H}amiltonian systems}, Comment. Math. Helv. \textbf{75} (2000), no.~2,
  271--283.

\bibitem{Zung-Poincare2002}
N.T. Zung, \emph{Convergence versus integrability in Poincaré-Dulac normal forms}, 
Math. Res. Lett. 9 (2002), 217-228.

\bibitem{Zung-FocusII2002}
N.T. Zung, \emph{{Another note on focus-focus singularities}}, Lett. Math.
Phys. \textbf{60} (2002), no.~1, 87--99.

\bibitem{Zung-Tedemule2003}
N.T. Zung, \emph{{Actions toriques et groupes d'automorphismes de
singularités de
  systèmes dynamiques intégrables}}, C. R. Acad. Sci. Paris \textbf{336}
  (2003), no.~12, 1015--1020.

\bibitem{Zung-Integrable2003}
N.T. Zung, {\it Symplectic topology of integrable Hamiltonian systems. II. Topological classification}, 
Compositio Math. 138 (2003), no. 2, 125–156. 
 
\bibitem{Zung-Birkhoff2005}
N.T. Zung, \emph{Convergence versus integrability in Birkhoff normal forms}, 
Ann. of Math. (2) 161 (2005), no. 1, 141–156.

\bibitem{Zung-Torus2006}
N.T. Zung, \emph{Torus actions and integrable systems},  in Topological Methods in the Theory of Integrable Systems,
Editors A.V. Bolsinov, A.T. Fomenko and A.A. Oshemkov, Cambridge Scientific Publications, 2006, 289--328.

\bibitem{Zung-Nondegenerate2012}
N.T. Zung, {\it Nondegenerate singularities of  integrable dynamical systems}, Ergodic Th. Dyn. Sys.
2014 (to appear, online first 2013).

\bibitem{Zung-OrbitalSmooth2012}
N.T. Zung, {\it Orbital linearizaton of smooth completely integrable dynamical systems}, preprint arxiv:1204.5701, 2012.

\bibitem{Zung-AA2014}
N.T. Zung, {\it Action-angle variables on Dirac manifolds}, preprint arxiv:2012 (submitted).

\bibitem{ZungMinh-2dim2013}
N.T. Zung, N.V. Minh, {\it Geometry of integrable dynamical systems on 2-dimensional surfaces}, Acta Math. Vietnam., 
38 (2013), Issue 1, 79--106.

\bibitem{ZungMinh_Rn2014}
N.T. Zung, N.V. Minh, \emph{Geometry of nondegenerate $\bbR^n$-actions on $n$-manifolds},
J. Math. Soc. Japan. 66 (2014), No. 3, 839--894.

\bibitem{ZungThien-SDS2014}
N.T. Zung, N.T. Thien, {\it Reduction and integrability of stochastic dynamical systems}, 
preprint, 2014.

}
\end{thebibliography}
\end{document}